\documentclass{article}
\usepackage{amsmath,amssymb,color}
\usepackage{graphicx}
\usepackage{hyperref}

\usepackage{epsfig}
\usepackage{fullpage}
\usepackage{cite}
\usepackage{palatino}
\numberwithin{equation}{section}

\newcommand{\qed}{\hfill$\blacksquare$}

\newcommand{\R}{\mathbb{R}}

\def\im{\mathop{\rm im}\nolimits}
\def\ind{\mathop{\rm ind}\nolimits}
\newcommand{\spanop}[1]{{\mathop{\rm span({#1})}}}

\newtheorem{thm}{Theorem}[section]
\newtheorem{lem}[thm]{Lemma}
\newtheorem{prop}[thm]{Proposition}

\newtheorem{define}[thm]{Definition}

\newtheorem{example}[thm]{Example}
\newtheorem{remark}[thm]{Remark}
\newtheorem{remarks}[thm]{Remarks}
\newtheorem{conjecture}[thm]{Conjecture}
\newenvironment{proof}{{\bf Proof.}  }{\hfill$\blacksquare$}

\newcommand{\g}{\gamma}

\newcommand{\e}{\epsilon}

\newcommand{\1}{\mathbf{1}}
\newcommand{\0}{\mathbf{0}}
\newcommand{\av}[1]{\left|{#1}\right|}

\renewcommand{\l}{\left}

\renewcommand{\r}{\right}

\newcommand{\be}{\begin{enumerate}}
\newcommand{\bi}{\begin{itemize}}
\newcommand{\ee}{\end{enumerate}}
\newcommand{\ei}{\end{itemize}}

\newcommand{\G}{{\Gamma}}
\newcommand{\Gp}{{\G_+}}
\newcommand{\Gm}{{\G_-}}
\renewcommand{\P}{\mathbb{P}}

\renewcommand{\phi}{\varphi}

\renewcommand{\L}{\mathcal{L}}
\newcommand{\M}{\mathcal{M}}

\newcommand{\ST}{{\mathcal{ST}}}
\newcommand{\tstar}{t^\star}

\begin{document}
\title{Spectral Theory for Networks with Attractive and Repulsive
  Interactions} \date{\today}

\author{Jared C. Bronski \\ University of Illinois \and Lee DeVille \\
  University of Illinois \\ }

\maketitle

\begin{abstract}
There is a wealth of applied problems that can be posed as a dynamical
system defined on a network with both attractive and repulsive
interactions. Some examples include: understanding synchronization
properties of nonlinear oscillator;, the behavior of groups, or
cliques, in social networks; the study of optimal convergence for
consensus algorithm; and many other examples.  Frequently the
problems involve computing the index of a matrix,~i.e. the number of
positive and negative eigenvalues, and the dimension of the kernel. In
this paper we consider one of the most common examples, where the
matrix takes the form of a signed graph Laplacian. We show that the
there are topological constraints on the index of the Laplacian matrix
related to the dimension of a certain homology group. In certain
situations, when the homology group is trivial, the index of the
operator is rigid and is determined only by the topology of the
network and is independent of the strengths of the interactions. In
general these constraints give upper and lower bounds on the number of
positive and negative eigenvalues, with the dimension of the homology
group counting the number of eigenvalue crossings. The homology group
also gives a natural decomposition of the dynamics into ``fixed''
degrees of freedom, whose index does not depend on the edge-weights,
and an orthogonal set of ``free'' degrees of freedom, whose index
changes as the edge weights change. We also present some numerical
studies of this problem for large random matrices.

\end{abstract}

\section{Introduction}\label{sec:intro}

\subsection{Problem Formulation}\label{sec:probform}
There are many applied problems that can ultimately be reduced to 
the question of understanding a dynamical problem on a network or graph. 
In these applications it is often important to understand the dynamical 
behavior of the evolution in terms of the topological properties of the graph. 

In this paper we consider a simple, connected, undirected edge-weighted 
graph $\G =(V(\G), E(\G))$ with vertex set $V(\Gamma)$ and edge set 
$E(\Gamma)$.  For each edge in $E(\Gamma)$ connecting vertex $i$ with 
vertex $j$ we associate a weight $\gamma_{ij},$ which is assumed to be 
non-zero but may take either sign. If there is no edge connecting 
vertices $i$ and $j$ the weight $\gamma_{ij}$ is understood to 
be zero.

For such a graph we define the {\em signed Laplacian matrix} $\L(\G)$ by
\begin{equation}\label{eq:Laplacian}
  \L(\G)_{ij} = \begin{cases} \g_{ij}, & i\neq j,\\
      -\sum_{k \neq i}\gamma_{ik},& i = j.\end{cases}   
\end{equation}
Note that  $\L(\G)$ is symmetric, so all eigenvalues of $\L(\G)$ are
real.
If the weights are all positive, $\gamma_{ij}>0,$ this is a standard 
graph Laplacian and as such the spectrum of ${\mathcal L}$ is 
well-understood: $\mathcal L$ is a negative semi-definite matrix, with the 
dimension of the kernel equal to the number of connected components of the 
graph $\G$. In many applications the weights are not guaranteed to be 
positive. In this case the matrix is no longer definite, and we are interested 
in determining the number of positive,
zero, and negative eigenvalues of $\L(\G)$, which we denote as
$n_+(\G), n_0(\G), n_-(\G)$ respectively.  

In particular, we would like to determine {\em topological} bounds on
these indices, i.e. conditions that depend only on the arrangement of
signs of the edge weights $\gamma_{ij}$, and not on their
magnitudes. We are aware of at least five applied problems that
motivate this question:

\begin{enumerate}

\item {\bf Stability of Fixed Points} Consider any dynamical system
  defined on the network.  Specifically given a graph $\G$, and
  symmetric coupling functions $\varphi_{ij}(\cdot) = \varphi_{ji}(\cdot)$
  associated to edges in the graph we define a dynamical system as
  follows
  \begin{equation}\label{eq:network}
    \frac{d}{dt} x_i = F_i({\mathbf x}) :=  \omega_i + \sum_{j\in\G}\varphi_{ij}(x_j-x_i).
  \end{equation}
  A well-studied example of this type of dynamical system is the
  Kuramoto oscillator~\cite{Kuramoto.86, Kuramoto.book, Kuramoto.91,
    S, Acebron.etal.05}, where we choose $\varphi_{ij}(\cdot) =
  \gamma_{ij} \sin(\cdot)$.
  
  To compute the stability index of a fixed point for the
  system~\eqref{eq:network}, i.e. a vector $\mathbf x$ with
  $F_i(\mathbf x) = 0$ for all $i$, we need to determine the index of
  the Jacobian $J$, where
  \begin{equation*}
    J_{ij} = \begin{cases} \varphi_{ij}'(x_j-x_i),& i\neq j,\\
       -\sum_{k} \varphi_{ik}'(x_k-x_i), & i = j.\end{cases}
  \end{equation*}
  The Jacobian $J$ is a graph Laplacian of the
  form~\eqref{eq:Laplacian}; thus, determining the stability indices
  for fixed points of~\eqref{eq:network} is equivalent to the problem
  studied here~\cite{Mirollo.Strogatz.2005, VO1, VO2, MS1, MS2,
    Bronski.DeVille.Park.2012}. When studying this dynamical system,
  the first object of study is always the stable points.  But, for
  example, if we consider such a system perturbed by a small
  stochastic process, then gaining a qualitative understanding of the
  dynamics requires that we identify all of the 1-saddles, i.e.  those
  points which are unstable but with exactly one unstable direction
  (e.g. see~\cite{NL-DV-2012}).

  The $\varphi'_{ij}$ terms in $J$ can be of either sign, making it
  natural to consider the case of a graph Laplacian with arbitrary
  signs on the weights. For a generic choice of coupling functions,
  $\varphi'_{ij}(x_i-x_j)$ is non-zero for all $i,j$, implying that
  the graph determining $J$ and the graph defined by the original
  interactions in~\eqref{eq:network} have the same underlying
  topology.

%
%

\item {\bf Stability of Neural Networks} If we consider any neural
  network system with both ``positive'' and ``negative feedbacks'',
  then the stability analysis reduces to a eigenvalue problem similar
  to ~\eqref{eq:Laplacian}. In biological applications the interaction
  strengths are both difficult to measure experimentally and (due to
  neural plasticity) very changeable. This makes it impractical to
  estimate the magnitudes of $\gamma_{ij}$ with any degree of
  reliability.  On the other hand the nature of the interaction
  (excitatory or inhibitory) is anatomical and, in general, will not
  change. Thus we have reliable experimental data on the signs of
  various connections, but very little reliable information on the
  magnitudes. This makes the idea of understanding the extent to which
  the dynamics is determined by the topology of the network a very
  attractive one.

\item {\bf Convergence of Consensus Algorithms} It was pointed out by
  Xiao and Boyd~\cite{Xiao.Boyd.04} that for some graphs, the optimal
  choice of weights for the convergence of a consensus algorithm uses
  negative weights.  In particular, given a weighted graph $\G =
  \{\gamma_{ij}\}$, define the {\em linear discrete-time consensus
    algorithm} given by
  \begin{equation}\label{eq:consensus}
    x_i(t+1) := \sum_j \gamma_{ij} x_j(t).
  \end{equation}
  The question posed in~\cite{Xiao.Boyd.04} is how one might optimize
  the choice of $\gamma_{ij}$, given the underlying graph topology and
  the constraint that $\1$ be a stable fixed point
  of~\eqref{eq:consensus}, to obtain the most rapid convergence to
  consensus (to wit, to make the Lyapunov exponent
  of~\eqref{eq:consensus} as small as possible).  It was observed
  there that there exist examples of graphs where the optimal choice
  involves negative weights---that is to say, that there are choices
  of weights that make~\eqref{eq:consensus} converge more rapidly than
  the most rapidly-mixing Markov chain.  This observation inspired a
  deluge of work (examples include~\cite{Hatano.Mesbahi.05,
    Olfati-Saber.etal.07, Kashyap.Basar.Srikant.07, Xiao.Boyd.Kim.07})
  on this fast convergence problem and it has been observed that the
  need for negative weights is typical in many contexts.

  Of course, one obvious constraint on determining the optimal choice
  of weights is that $n_+(\G) = 0$ so that we obtain convergence at
  all.  In the work mentioned above, this was always obtained by
  solving a semidefinite programming problem over the set of all
  weights associated to a particular unweighted graph $\G$.  The
  results in the current work give topological
  (i.e. weight-independent) bounds for $n_+(\G)$.

\item {\bf Clustering in Social Networks} Questions of this sort arise
  in the study of social networks.  In the classic work of Hage and
  Harary~\cite{Hage.Harary.book}, matrices of the
  form~\eqref{eq:Laplacian} modeled the interaction of tribal groups
  within an alliance in New Guinea.\footnote{It should be noted that
    Hage and Harary were concerned with the question of balance in
    signed graphs, a very different question from the ones we consider
    here.} In this work the underlying graph has sixteen vertices,
  representing the different tribal units, with edges represent
  relations between different tribal groups.  These relations can be
  friendly (``rova'') or antagonistic (``hina'') corresponding to
  $\gamma_{ij}>0$ and $\gamma_{ij}<0$ respectively. Another
  anthropological example is the SlashDot Zoo, a social network
  associated to the website SlashDot~\cite{SlashDot}.  On this website
  participants can label each other as friend or foe.\footnote{In this
    example the edges are directed, leading to a non-symmetric
    Laplacian matrix.  As we do not consider non-symmetric graphs in
    this work, we will need to consider a symmetrized version of the
    SlashDot Zoo in the applications below.  However, many social
    networks (e.g. Facebook) are symmetric social networks by
    definition.}  In the context of social network models the index
  $n_+(\G)$ indicates the tendency of the network to separate into
  mutually antagonistic subgroups.

\item{\bf Resistor Networks with Negative Resistance} One intuitively
  appealing way to think about a signed graph Laplacian is as a
  network of resistors where some of the edges have negative
  resistance.  While perfect negative resistors cannot be implemented
  with passive components there are many nonlinear components whose
  current-voltage curves are non-monotone, and have a region where the
  current is a decreasing function of the voltage. In the region where
  the current is a decreasing function of current the component has
  negative (differential) resistance.  Some examples of such
  components include tunnel and Gunn diodes, neon lamps, and certain
  kinds of tubes. Circuits built from these components often exhibit
  multi-stability, and the stability of a given solution is governed
  by a system of the form \eqref{eq:Laplacian}.  Alternatively a
  negative resistance can be implemented using active circuit elements
  with gain, such as an operational amplifier (op-amp).
\end{enumerate}

On a somewhat different note we should remark that signed graphs and
signed Laplacians arise very naturally in knot theory.  One classical
result of this kind is a procedure for associating a quadratic form to
a knot originally due to Goeritz~\cite{Georitz.1933,
  Gordon.Litherland.1978}. In Goeritz's construction one considers a
planar projection of the knot, with regions of plane alternately
colored black and white. Two regions of given color are connected by
an edge if they share a crossing. The weight attached to the edge is
$+1$ if the crossing is left-handed and $-1$ if the crossing is
right-handed.  The reduced determinant --- the product over the
non-zero eigenvalues of the associated graph Laplacian --- can be
shown to be a knot invariant.  Related results include a construction
by Kauffman of a Tutte polynomial for signed graphs~\cite{kauffman.89}
which specializes to known invariants like the Jones~\cite{jones.85}
and Kauffman bracket polynomials~\cite{kauffman.87}. This construction
was further generalized to matroids by
Zaslavsky~\cite{zaslavsky.92}. Signed graphs and their Laplacians have
also been studied in the graph theory community independent of their
connection to knot theory (see~\cite{Zaslavsky.82, Zaslavsky.coloring,
  Zaslavsky.02} and following, also~\cite{Hou.05}) but the the reader
should be aware that there are several different generalizations of
the Laplacian to the case where negatively weighted edges are
allowed. In one variation definition the diagonal entries are taken to
be minus the sum of the absolute values of the edge weights. In this
case the matrix is typically not zero-sum, and is negative
semi-definite.  Clearly the spectral questions are different and there
is no obvious correspondence between the two.

To make clear the kinds of questions we want to ask, and the kind of phenomena we would like 
to understand, we begin with an example of two networks. 

\begin{figure}[ht]
\begin{centering}
\includegraphics[height=2.5in]{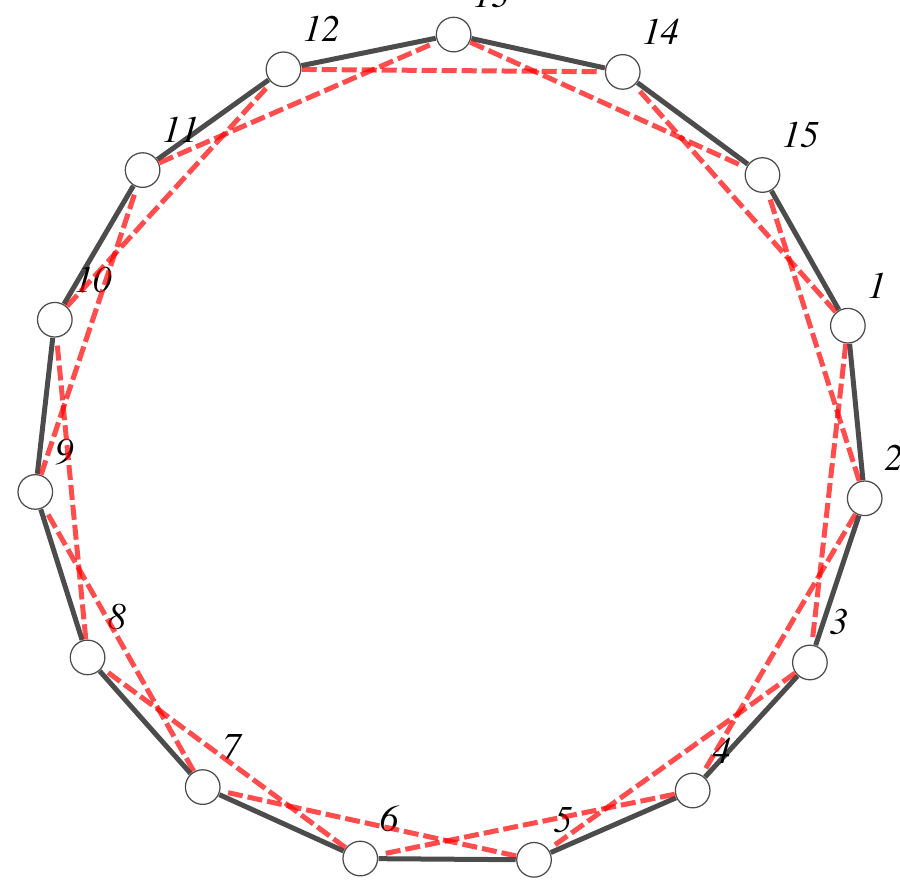} \includegraphics[height=2.5in]{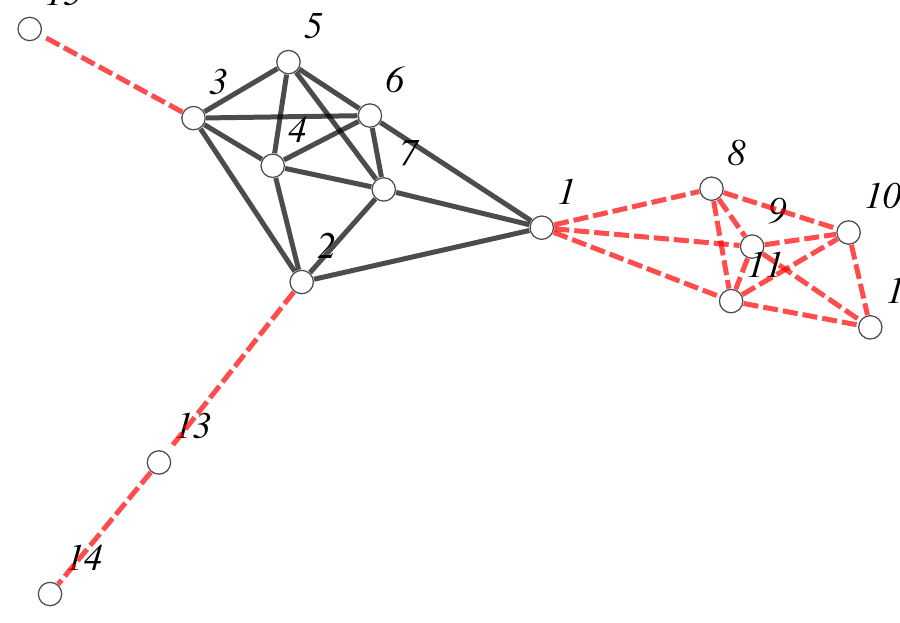}
\caption{Two graphs on 15 nodes.  Here we are represented positively
  weighted edges with solid black lines, and negatively weighted edges
  by dashed red lines. }
\label{fig:motivate}
\end{centering}
\end{figure}
 
\begin{example}
  Figure \ref{fig:motivate} depicts two signed graphs representing
  networks.  in this figure dashed lines (red online) represent
  negative/excitatory/hina interactions and the solid line represent
  positive/inhibitory/rova interactions. Each of these graphs has
  fifteen vertices, fifteen hina edges and fifteen rova edges. The
  dynamics, however, are typically very different. We claim that there
  is a marked difference between the way mathematicians tend to view
  these problems and the way other network scientists view them, and
  that these differences are summarized in the graphs above. The first
  graph is more symmetric, having as its automorphism group (in the
  absence of weights) the dihedral group $D_{15}$. The automorphism
  group of the second graph is trivial, consisting of only the
  identity. Since symmetry reduction is one of the most common and
  most powerful techniques in applied mathematics this suggests that
  the dynamics of the first network should be easier to
  understand. Many network scientists, however, would claim that the
  second network is simpler and easier to understand. They would argue
  that this network can clearly be separated into four functional
  units. The largest functional unit consists of vertices one through
  seven, which interact through mutual inhibition. Hanging off of this
  structure are three small units: the first consisting of vertices
  one and eight through twelve; the second of vertices two, thirteen
  and fourteen; and the third of vertices three and fifteen. These
  vertices in these smaller units act on each other via excitation and
  couple to the main unit through the shared vertices. The goal of
  this paper is to mathematize this intuition about the structure of
  the network. We will show that, at least for the kinds of questions
  asked in this paper, the second point of view is the more fruitful
  and amounts to an observation about a particular homology group of
  the graph. The first network has the property that a certain
  homology group has the maximum possible dimension. This, as we will
  show, implies that the spectrum of the Laplacian is essentially
  arbitrary: it has one zero eigenvalue, and the signs of the other
  fourteen eigenvalues can be chosen arbitrarily with an appropriate
  choice of weights. In the second network, on the other hand, the
  analogous homology group is trivial. This implies that the spectrum
  is rigid: the Laplacian always has eight negative eigenvalues, one
  zero eigenvalue, and six positive eigenvalues regardless of the
  choice of weights on the edges.\footnote{Note that there is no
    obvious relationship between the size of the automorphism group
    and the size of the homology group. In these examples they are
    both large or both small, but it is easy to construct examples
    where one is small and the other large.}
\end{example}

In this paper we give the best possible bounds on
$n_{-}(\G),n_{0}(\G),n_{+}(\G)$ involving only topological information
--- connectivity of the graph and the sign information on the edge
weights. For a graph with $N$, vertices the difference between the
upper and lower bounds is an integer that can vary between $0$ and
$N-1$, depending on the topology of the graph.  This integer
represents the dimension of a certain homology group, and counts the
number of possible eigenvalue crossings from the left to the right
half-plane. The examples above represent extreme cases. In the first
example the homology group has fourteen generators and there are
fourteen possible eigenvalue crossings. In the second example the
homology group is trivial, consisting only of the zero element, and
there are no eigenvalue crossings. This homological construction gives
a natural splitting of the vector space into a ``fixed'' subspace,
where there cannot be an eigenvalue crossing, and a ``free'' subspace,
where all of the eigenvalue crossings occur. We also show that these
bounds are strictly better than those implied by the Gershgorin
theorem. Finally we will conclude with some numerical experiments and
examples.

\section{Main Theorem}\label{sec:main}

\subsection{Preliminaries}\label{sec:gt}

The main theorem, Theorem \ref{thm:main}, gives tight upper and 
lower bounds on the number of positive, negative, and zero eigenvalues. 
To state the main theorem of the paper, we first present a few
definitions. 

\begin{define}
 Given a graph $\G$, we define the two subgraphs $\Gp$
  (resp.~$\Gm$) to be the subgraphs with the same vertex set as $\G$
  ($V(\G)=V(\Gp)=V(\Gm)$) together with the edges of positive
  (resp. negative) weights:
  \begin{equation*}
    E(\Gp) = \{ \e \in E(\G) | \gamma_{ij} >0\} ,\quad E(\Gm) = \{ \e \in E(\G) | \gamma_{ij} <0\}.
  \end{equation*}
  Moreover, these graphs inherit the weighting from the original $\G$,
  i.e.
  \begin{equation*}
    \l(\Gp\r)_{ij} = \max(\gamma_{ij},0),\qquad \l(\Gm\r)_{ij} = \min(\gamma_{ij},0).
  \end{equation*}
We further adopt the following definitions and notations:
\begin{itemize}

\item We use the notation $c(\cdot)$ to represent the number of
  components of a graph. 

\item We let $\Gamma_+^{(i)}$, $i \in (1\ldots c(\Gp))$ denote the
  $i^{th}$ component of $\Gp$, and similarly $\Gamma_-^{(i)}$.

\item Given a graph $\G$ we let $\ST(\G)$ denote the set of all
  spanning trees of $G$. More generally we let $\ST_k(\G)$ denote the
  set of all spanning trees of $\G$ having exactly $k$ edges in
  $\Gm$. Note that $\cup_{k=0}^{N-1} {\mathcal ST}_k = {\mathcal ST}$
  and ${\mathcal ST}_k \cap {\mathcal ST}_{k'} = \emptyset$ if $k \neq
  k'$.

\item We define the {\em flexibility} of a weighted graph as the
  number
  \begin{equation*}
    \tau(\G) := \av{V(\G)} - c(\Gm) - c(\Gp) + 1.
  \end{equation*}
  If $\tau(\G) = 0$, then we say that $\G$ is {\em rigid}.  We show
  below that the flexibility is always a non-negative number.

\item As mentioned above, for any weighted graph $\G$, we define the three
  indices $n_0(\G), n_-(\G), n_+(\G)$ as the number of zero, negative,
  and positive eigenvalues of $\L(\G)$. We will occasionally use the notation ${\rm ind}(\Gamma)$ 
to represent the triple ${\rm ind}(\Gamma) = \left(n_-(\G), n_0(\G), n_+(\G)\right).$

\end{itemize}
\end{define}

\begin{remark}
  Any graph Laplacian as defined in~\eqref{eq:Laplacian} has all row
  sums equal to zero, so that $\L(\G)\1 = \0$, and one necessarily has
  $n_0(\G)\ge 1$.  It is well-known\cite{Cvetkovic.etal.book,
    Chung.book} that if all the weights $\g_{ij}\ge 0$, then the graph
  Laplacian is negative semi-definite, with $n_0(\G)=c(\G)$ and thus
  $n_- = |V(\G)| - c(\G).$ In particular, if $\G$ is connected with
  positive weights, then
  \begin{equation*}
    n_+(\G) = 0,\quad n_0(\G) = 1,\quad n_-(\G) = \av{V(\G)} - 1.
  \end{equation*}
  This is no longer true when the edge-weights are allowed to be
  negative --- the Laplacian matrix of a connected graph can have
  multiple zero and positive eigenvalues.  
\end{remark}

\begin{lem}\label{lem:flexibility}
Every signed graph $\G$ satisfies the inequality
\begin{equation}\label{eq:deflem}
  c(\Gp) + c(\Gm) \leq |V(\G)|+c(\G).
\end{equation}
If $\G$ is connected, writing $|V(\G)|=N$ gives
\begin{equation*}
  c(\Gp) + c(\Gm) \leq N + 1.
\end{equation*}
From this, it follows that the flexibility $\tau(\G)$ of any graph is a non-negative
integer.
\end{lem}   

\begin{proof}
  We first note that if $\Gm$ contains a subforest $T$ with $\ell$ edges,
  then $c(\Gm) \le \av{V} - \ell$.

  We now define $\widetilde{\G} = (\widetilde{V},\widetilde{E})$ to a
  quotient graph
  of $\G$, where $\widetilde{V}$ are the connected components of $\Gp$
  and $(a,b)\in\widetilde{E}$ iff there is at least one edge in $\Gm$
  from a vertex in component $a$ to a vertex in component $b$.

  Since $\G$ is connected, so is $\widetilde{\G}$.  Consider any
  spanning tree $\mathcal{T}$ of $\widetilde{\G}$.  By definition,
  this contains $c(\Gp) - 1$ edges, since it is a tree on $c(\Gp)$
  vertices.  Now consider this tree ``lifted'' into $\G$ where, for
  every edge in $\widetilde{\G}$ of the form $a\leftrightarrow b$,
  choose one edge in $E$ that connects component $a$ to component $b$.
  This must be a subforest of $\Gm$, since it cannot contain any
  cycles.  Therefore $c(\Gm) \le \av{V} - (c(\Gp) - 1)$ and we are
  done.

  Finally, if $\G$ is not connected, let $\G^{(1)},\dots,\G^{(c(\G))}$
  be the connected components of $\G$, and define $\G_+^{(i)},
  \G_-^{(i)}$ in the obvious manner.  By assumption, we have that
  \begin{equation*}
    c(\G_+^{(i)}) + c(\G_-^{(i)}) \le \av{\G^{(i)}} + 1
  \end{equation*}
  It is not hard to see that 
  \begin{equation*}
    \sum_{i=1}^{c(\G)} c(\G_+^{(i)}) = c(\Gp),\quad \av{\G} = \sum_{i=1}^{c(\G)}\av{\G^{(i)}},
  \end{equation*}
  and from this~\eqref{eq:deflem} follows.  
\end{proof}

\begin{remark}
  Since $\tau(\Gamma)$ is a counting number, the obvious next thing to
  determine is what it is that it counts.  We will show in
  Section~\ref{sec:MV} from the consideration of a particular
  Mayer-Vietoris sequence that $\tau(\G)$ is the dimension of a
  certain homology group.  This will be useful because it gives a
  relationship between the flexibility of a graph and cycles of a
  certain type.  However, it is convenient to have the above proof,
  which is self-contained and purely graph-theoretic, at the current
  time.
\end{remark}

The main machinery that we will use in this paper is the celebrated
Kirchhoff matrix tree theorem~\cite{kirchhoff.1847,kirchhoff.1847E,Chaiken.82}.   To state it,
we first present some notation, which essentially follows that of Tutte\cite{Tutte}:

\begin{define}
  \begin{itemize}

  \item Let ${T}$ be a tree, then we define $\pi({T})$ to be the product over the edge weights in the tree
  \begin{equation}\label{eq:defofpi}
    \pi({T}) = \prod_{(i,j)\in E({T})} \gamma_{ij}.
  \end{equation}

\item Let $\G$ be a weighted graph with $\av{V(\G)}=N$, and $\L(\G)$
  be its graph Laplacian.  We know that $\L(\G)$ has a zero
  eigenvalue, and therefore $\det(\L(\G))= 0$.  Order the $n$
  eigenvalues of $\L(\G)$ so that $\lambda_1=0$, then we define
  \begin{equation}\label{eq:defofM}
    \M(\G) = \frac{(-1)^{N-1}}N  \prod_{i=2}^N \lambda_i.
  \end{equation}
In other words $\M(\G)$ is (up to the multiplicative prefactor) the linear term 
in the characteristic polynomial of the Laplace matrix. 
Note that $\M(\G) \neq 0$ iff $0$ is a simple
  eigenvalue of $\L(\G)$.  More generally, the multiplicity of a zero of
  $\M(\G)$ is one fewer than the multiplicity of zero in the
  characteristic polynomial of $\L(\G)$. $\M(\G)$ is a reduced
  determinant which is zero iff $\L(\G)$ has a non-simple zero
  eigenvalue.
\end{itemize}
\end{define}

With this notation the Kirchhoff matrix tree theorem can be stated as follows:

\begin{lem}[Weighted Matrix Tree Theorem]\label{lem:mtt}
  Let $\G$ be a connected, weighted graph, and $\ST(\G)$ the
  set of all spanning trees of $\G$.  Then
  \begin{equation}\label{eq:MTT}
    \M(\G) = \sum_{T\in\ST(\G)} \pi(T).
  \end{equation}
\end{lem}

\begin{remark}
  This is Theorem VI.29 in the text of Tutte~\cite{Tutte}: a proof is
  provided there.  Notice that if all of the edge weights are
  non-negative, then the sum in~\eqref{eq:MTT} is a sum of positive
  terms. This is an alternate proof that the kernel of a graph
  Laplacian with positive weights is simple for a connected graph.
  However, once we allow negative weights, the sum on the right-hand
  side can have cancellations, and this is the major difficulty in
  understanding the spectral properties of graphs with negative
  weights.
\end{remark}

\subsection{Statement and proof of main theorem}\label{sec:proof}

We begin by stating one of the main results of this paper.

\begin{thm}\label{thm:main}
  Let $\G$ be a connected signed graph, and
  $n_{-}(\G),n_0(\G),n_+(\G)$ be the number of negative, zero, and
  positive eigenvalues respectively. Then for any choice of weights
  one has the following inequalities:
  \begin{equation}\label{eq:bounds}
    \begin{split}
      c(\Gp)-1 \le &n_+(\G) \le  N-c(\Gm), \\
      c(\Gm)-1 \le &n_-(\G)  \leq  N-c(\Gp), \\
      1\leq &n_0(\G) \leq  N+2-c(\Gm)-c(\Gp).  
    \end{split}
  \end{equation}

Further these bounds are tight: for any given graph there exist 
open sets of weights giving maximal number of negative eigenvalues 
\begin{align*}
  n_+(\G) &= c(\Gp)-1, \\
  n_-(\G) &= N-c(\Gp),\\
  n_0(\G) &= 1.
\end{align*}  
as well as open sets of weights giving the maximal number of positive eigenvalues
\begin{align*}
  n_+(\G) &= N-c(\Gm), \\
  n_-(\G) &= c(\Gm)-1,\\
  n_0(\G) &= 1.
\end{align*}  
\end{thm}

\begin{remarks}
Notice that in each inequality in~\eqref{eq:bounds} the difference between the upper and
lower bound is exactly the flexibility of the graph $\tau(\G)$ .  This shows that
$\tau(\G)$ counts the number of potential eigenvalue crossings. More precisely the
theorem shows that there are $c(\Gp)-1$ eigenvalues which are always
negative, $c(\Gm)-1$ eigenvalues which are always positive, and
$\tau(\G) = N+1-c(\Gm)-c(\Gp)$ eigenvalues depend on the choice of
weights. For rigid graphs there are no eigenvalue crossings and the
index is fixed regardless of the choice of weights (thus inspiring the
terminology ``rigid''). 

In general the index will be constant with $n_0(\G)=1$ on open sets,
separated by codimension one sets where   $n_0(\G)=2$. One expects 
that  $n_0(\G)>2$ only on sets of higher codimension. 

In the context of network models such as~\eqref{eq:network} described
above, the theorem shows that a necessary condition for stability of a
phase-locked state is that between any two oscillators there exists a
path such $\phi^\prime_{ij}(x_i-x_j)>0$ for all edges on the
path. While the necessity of such a condition seems physically obvious
we are unaware of a previous proof of this.
\end{remarks}

The main idea of the proof is to consider a one-parameter family of
weighted Laplacian matrices. We then compute a polynomial associated
to the graph, the zeroes of which detects eigenvalue crossings. We
then show that this polynomial has exactly $\tau$ roots in the
positive half-line, and that the multiplicity of these roots 
is equal to the number of eigenvalues crossing from the left to the right half-plane 
at that parameter value

\begin{define}
 Given a weighted graph $\G$ we define a one-parameter family of
 weighted graphs $\G(t)$ as follows: the weights of $\G(t)$ are
 related to those of $\G$ by
\begin{equation*}
  \gamma_{ij}(t) := \begin{cases} \gamma_{ij}, & \gamma_{ij}>0,\\ t\cdot\gamma_{ij}, & \gamma_{ij}<0,\end{cases}
\end{equation*}
or, more compactly, $\G(t) = \Gp + t\Gm$.  Obviously $\G = \G(1)$.

We also recall Lemma~\ref{lem:mtt}, the weighted matrix tree theorem,
and define
\begin{equation}\tag{\ref{eq:MTT}} 
  \M(\G(t)) = \sum_{T\in\ST(\G)} \pi(T).
\end{equation}
\end{define}

The next observation is that  $\M(\G(t))$ is a polynomial of a very special form:

\begin{lem}
\label{lem:poly}
$\M(\G(t))$ is a polynomial in $t$ which takes the following form 
 \begin{equation}\label{eq:pg}
  \M(\G(t)) = \sum_{k=c(\Gp)-1}^{N-c(\Gm)} a_k (-t)^k,
\end{equation}
where the coefficients $a_k$ are given by 
\[
 a_k = \sum_{T \in \ST_k(\G)} |\pi(T)|.
\] 
All of the $a_k$ appearing in~\eqref{eq:pg} are nonnegative; moreover,
the first and last coefficients, $a_{c(\Gp)-1}$ and $a_{N-c(\Gm)}$,
are strictly positive.
\end{lem}
 
\begin{proof}
  Clearly $\M(\G(t))$ is a polynomial in $t$, since it is given by
  sums and products of terms each of which is at most linear in
  $t$. Since there is one power of $t$ associated with each negatively
  weighted edge and $\ST = \cup_{k=0}^{N-1}\ST_k $ it follows that
  \[
  \M(\G(t)) = \sum_{k=0}^{N-1} a_k (-t)^k. 
  \]
  with $a_k$ defined as above. 

  Next we note that $\ST_k$ is empty for $k < c(\Gp)-1$ and non-empty
  for $k=c(\Gp)-1$. To see this note that, since $\Gp$ has $ c(\Gp)$
  components we need at least $c(\Gp)-1$ negative edges to connect
  them, and that there is at least one way to construct a spanning
  tree with $c(\Gp)-1$ negative edges: we first construct a spanning
  tree on each of the $c(\Gp)-1$ components of $\Gp$ using positive
  edges, and then connect them with $c(\Gp)-1$ negative edges. This
  can always be done since the graph is assumed to be connected. Since
  $a_{c(\Gp)-1}$ is a sum over a non-empty set of positive terms it is
  positive. The upper bounds follow from the dual argument: reversing
  the roles of $\Gp$ and $\Gm$ shows that any spanning tree must have
  at least $c(\Gm)-1$ positive edges. Since any spanning tree has
  exactly $N-1$ edges there are at most $N-1-(c(\Gm)-1)=N-c(\Gm)$
  negative edges.

\end{proof}

\begin{remark}
  We note that the polynomial $\M(\G(t))$ is strongly reminiscent of
  other graph polynomials such as the chromatic, rank and Tutte
  polynomials, which have definitions in terms of sums over spanning
  trees. We will show later in the paper that $\M(\G(t))$ satisfies a
  contraction-deletion relation similar to that satisfied by other
  graph polynomials.
\end{remark}

\begin{lem}\label{lem:rr}
 The roots of the polynomial $\M(\G(t))$ are real and non-negative.
\end{lem}

\begin{proof}
 This follows from the observation that the roots of the polynomial are exactly the 
eigenvalues of a generalized symmetric eigenvalue problem. Note that a root of the polynomial 
 $\M(\G(t))$ corresponds to a solution of 
\[
 {\mathcal L}_+ v = -t {\mathcal L}_- v  
\]
where $v$ can be assumed to be orthogonal to $(1,1,1,\ldots,1).$ A
standard result in the theory of the generalized symmetric eigenvalue
problem (gsep) is that a sufficient condition for the problem to have
all real eigenvalues is that there exists a linear combination of
$L_+$ and $L_-$ that is strictly positive definite. For completeness
we give a short proof of this here. First note that the gsep $A v =
\lambda B v$ has real eigenvalues if either $A$ or $B$ is strictly
positive definite. If $B$ is strictly positive definite then the above
problem is self-adjoint under the inner product $\langle v,v\rangle=
v^t B v$, proving reality of the eigenvalues. If $A$ is strictly
positive definite a similar calculation holds with $\lambda$ replaced
by $\lambda^{-1}.$ Next note that there is an equivariant action of
$GL(2,R)$ on the generalized symmetric eigenvalue problem, as follows:
if $(A,B)$ is a pair with eigenvalue $\lambda$ $A v = \lambda B v $
then for any scalars $\alpha,\beta,\gamma,\delta$ with
$\alpha\delta-\beta\gamma\neq 0$ the pair $(\alpha A + \beta B, \gamma
A + \delta B)$ has an eigenvalue $\mu =\frac{\alpha \lambda + \beta
}{\gamma+\delta \lambda}$ and the same eigenvector $v$
\begin{equation}\label{eq:ab}
  (\alpha A + \beta B) v = \mu (\gamma A + \delta B) v. 
\end{equation}
Thus if there exists a positive linear combination of $A$ and $B$ then
the gsep $Av = \lambda B v$ has only real eigenvalues. In the case of
$L_+$ and $L_-$ it is clear that (for instance) $L_+ - L_-$ is
strictly positive definite on $(1,1,1,\ldots,1)^\perp$, as it is the
graph Laplacian for a connected graph with positive weights. Therefore
${\mathcal L}_+ v = -t {\mathcal L}_- v $ has only real roots.

For more information on the gsep see the review paper of Parlett\cite{Parlett.1991}, 
Theorem 1 in the  paper of Crawford\cite{Crawford.1976,Crawford.Errata.1978} or 
chapter IX \S 3 of the text of Greub\cite{Greub}, for a proof (due to Milnor) of a somewhat 
stronger 
result.

That none of the eigenvalues can be negative is clear from the form of the polynomial: since 
the coefficients alternate in sign $p(t)$ is obviously non-zero for $t<0$. 
    
\end{proof}


\begin{remark}\label{rem:logconcave}
 The fact that $P_\G(t)$ has only real roots implies, via Newton's inequality, 
that the sequence $\{a_k\}_{k=0}^N$ is log-concave
 \[
  a_{k+1} a_{k-1} \leq a_k^2.
 \]
In the special case where the weights of the edges are all $\pm 1$ then the
coefficients $a_k$ are integers which count the number of spanning trees 
of $\G$ having exactly $k$ edges in $\Gm.$  
A large number of other combinatorial sequences share this property. 
See the review papers of Stanley\cite{Stanley} or Brenti\cite{Brenti} for details. The analogous 
problem of the log-concavity of the coefficients of the chromatic polynomial, a much more 
difficult problem, was a long-standing conjecture that has recently been 
established by Huh\cite{Huh,Huh.Katz.12}. 

\end{remark}

Next we show that the multiplicity of the zeroes of the polynomial 
$\M(\G(t))$ is equal to the dimension of the kernel of 
${\mathcal L}(t)|_{(1,1,1,\ldots,1)^\perp}.$ First a preliminary lemma:

\begin{lem}
 The eigenvalues $\lambda_i(t)$ corresponding to eigenvectors orthogonal to $(1,1,1,\ldots,1)$  are non-decreasing 
functions of $t$ that cross zero transversely: if $\lambda_i(t)=0$ then $\lambda_i^\prime(t) > 0.$
\label{lem:transverse}
\end{lem}
\begin{proof}
 The fact that $\lambda_i(t)$ are non-decreasing follows immediately from the fact that ${\mathcal L}_+$
is a positive semi-definite matrix. The transversality follows from a perturbation argument and 
elementary topological considerations as follows. If $\lambda_i(t)$ vanishes at $t=t^*$ then from 
degenerate perturbation theory\cite{Kato} we have that $\lambda_i^\prime(t^*)$ is equal to one of the 
eigenvalues of the matrix 
\[
 {\mathcal L}_+ |_{\ker({\mathcal L}(t^*))}
\]
The matrix ${\mathcal L}_+$ is positive semi-definite, and thus
$\lambda_i^\prime(t^*)\geq 0.$ To see the strict inequality we note
that in order for ${\mathcal L}_+ |_{\ker({\mathcal L}(t^*))}$ to have
a zero eigenvalue there is necessarily a vector in $\ker({\mathcal
  L}_+) \cap \ker({\mathcal L}_-).$ Such a vector would be constant on
components of $\Gp$ { and } $\Gm$ and thus, by the connectedness
assumption, on all of $\G$. The only vectors in $\ker({\mathcal L}_+)
\cap \ker({\mathcal L}_-)$ are thus multiples of $(1,1,1,\ldots,1),$
and so any other zero eigenvalue must cross through the origin
transversely.
\end{proof}

\begin{lem}
The dimension of the kernel of ${\mathcal L}(t^*)$ restricted to $(1,1,1,\ldots,1)^\perp$ 
is equal to the multiplicity of $t^*$ as a root of $\M(\G(t)).$  
\end{lem}
\begin{proof}
The polynomial $\M(\G(t))=\prod_{i=1}^{N-1}\lambda_i(t)$ where, from the above, the 
$\lambda_i(t)$ have only simple roots. Thus the multiplicity of a root of $\M(\G(t))$ 
is equal to the number of $\lambda_i(t)$ that vanish there. 
\end{proof}

We are now in a position to prove Theorem \ref{thm:main}. \\

{\noindent\bf Proof of Theorem \ref{thm:main}.}  The matrix $\L(\Gm)$ is
  positive semi-definite so that the eigenvalues of $\L(\G(t))$ are
  non-decreasing functions of $t$. For $t=0$ the matrix ${\mathcal
    L}(t)$ is a graph Laplacian with $c(\Gp)$ components, so it has a
  $c(\Gp)$ dimensional kernel and an $N-c(\Gp)$ dimensional negative
  definite subspace. By lemma \ref{lem:transverse} $c(\Gp)-1$ of these
  zero eigenvalues cross transversely into the positive half-line, so
  for $t$ small and positive the index of $\L(\G(t))$ is
  $(N-c(\Gp),1,c(\Gp)-1).$ By lemma \ref{lem:rr} the crossing
  polynomial has exactly $\tau$ roots on the open positive half-line,
  and each root of the crossing polynomial corresponds to an
  eigenvalue crossing from the left half-line to the right, so one has
  exactly $\tau$ eigenvalue crossings. This gives
\begin{align*}
 &c(\Gp)-1 \leq n_+(\Gamma) \leq N - c(\Gm) \\
& c(\Gm)-1 \leq n_-(\Gamma) \leq N - c(\Gp).
\end{align*}
When $t$ is small and positive the lower holds for $ n_+(\Gamma)$ and the upper bound for 
$ n_-(\Gamma)$, and vice-versa when $t$ is large.
\qed

In the next section we consider more detailed asymptotics on the eigenvalues in the limits 
$t\ll 1$ and $t \gg 1$.

\subsection{Topological characterization of the flexibility}\label{sec:MV}

The quantity $\tau$ is a measure of the flexibility or rigidity of the
network dynamics: it is a measure of the number of eigenvalues that
can cross from the right half-plane to the left half-plane as the
weights of the connections are varied. It turns out that $\tau$ admits
a simple interpretation in terms of the topology of the graph via the
Mayer-Vietoris sequence.  This, in turn, will provide a great deal of
insight into the physics of the problem, allowing one to identify
important structures in the network. We begin with a definition:
 
\begin{define}
Consider the following map $\partial$ from $H_1(\Gamma)$, the space of
cycles in the graph, to $\R^N$. For each closed cycle $\gamma \in
H_1(\Gamma)$ we associate the following vector $v^\gamma =
\partial(\gamma) \in \R^N:$
\begin{itemize} 
\item For each time time the cycle enters vertex $i$ from a
  negatively weighted edge and exits through a positively weighted
  edge, $v^\gamma_i$ increases by one.
\item For each time time the cycle enters vertex $i$ from a
  positively weighted edge and exits through a negatively weighted
  edge, $v^\gamma_i$ decreases by one.
\item For vertices not on the cycle, or for vertices where the
  cycle enters and exits through edges of like weights, $v^\gamma$ is
  zero.
\end{itemize}
\end{define}

It is clear that this map is linear, and that the image is an additive
group.  It should also be clear that cycles that remain entirely in
edges of one type are in the kernel of this map and that only cycles
with both types of edges give rise to nontrivial vectors $v^\gamma.$
We refer to these cycles as ``cycles of mixed type''.

\begin{example}\label{exa:homology}
The following illustrates this map for two different graphs. The first
graph has two cycles of mixed type. The first ($\gamma_1$) is $3 \rightarrow 1
\rightarrow 2 \rightarrow 3$ and the second ($\gamma_2$) is $3 \rightarrow 5
\rightarrow 4 \rightarrow 3.$ This gives the additive group as linear
combinations of $v^{\gamma_1}=(-1,1,0,0,0)$ and $v^{\gamma_2}=(0,0,0,1,-1).$
\begin{figure}[ht]
\begin{centering}
\includegraphics[width=3in]{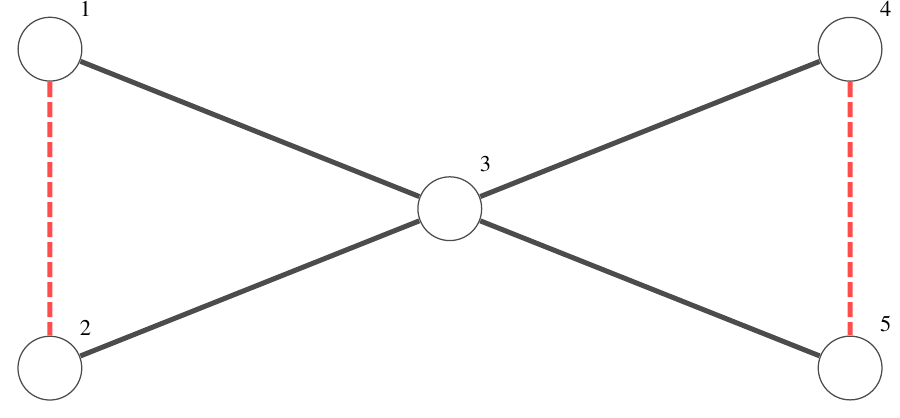} \includegraphics[width=3in]{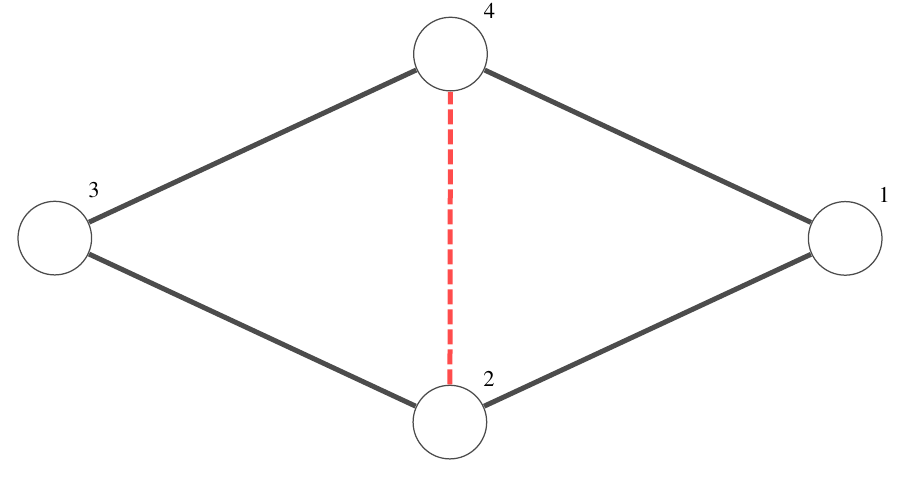}
\caption{The two graphs referenced in Example~\ref{exa:homology}.}
\label{fig:homology}
\end{centering}
\end{figure}
The second graph has only one circuit of mixed type. The cycles
$4\rightarrow 2\rightarrow 1 \rightarrow 4$ and $ 4\rightarrow
3\rightarrow 2\rightarrow 4$ sum to $4 \rightarrow 3 \rightarrow 2
\rightarrow 1 \rightarrow 4$. Since this is a circuit of all positive
edges it is in the kernel of $\partial$ and thus the second circuit is
the inverse of the first. This gives the additive group as vectors of
the form $k(0,1,0,-1)$.
\end{example} 

\begin{lem}
The flexibility 
\[
\tau = N + 1 - c(\Gamma_+)-c(\Gamma_-) 
\]
is equal to the the dimension of the group of cycles of mixed type. 
\end{lem}

\begin{proof}
  This is a straightforward application of the Mayer-Vietoris
  sequence~\cite[\S 25]{Munkres.book}. We have the exact sequence
\[
H_1(\G) \overset{\partial}\rightarrow H_0(\Gp\cap\Gm)
\overset{\alpha}\rightarrow H_0(\Gp)\oplus H_0(\Gm)
\overset{\beta}\rightarrow H_0(\G).
\]
The graph is assumed to be connected so $\dim(H_0(\G))=1$,
$\dim(H_0(\Gp)\oplus H_0(\Gm))=c(\Gp)+c(\Gm)$ and, since $\Gp$ and
$\Gm$ have no common edges, $\dim(H_0(\Gp \cap \Gm))=N$. The exactness
implies that $\dim(\ker(\beta))=\dim(\im(\alpha))=c(\Gp) + c(\Gm)-1$
and $\dim(\im(\partial))=\dim(\ker(\alpha))=N-(c(\Gp) + c(\Gm)-1) =
N+1-c(\Gp)-c(\Gm).$ Note that the last quantity is exactly the
flexibility. Thus the flexibility is equal to the number of linearly
independent cycles of mixed type.

\end{proof}

This construction shows that the dimension of the group of mixed
cycles is equal to the number of eigenvalues which cross from the left
to the right-half plane. It would be more satisfying to give an
explicit bijection between the basis of mixed cycles in the graph and
the eigenvalues which cross over, or the associated eigenspaces.  Of
course we cannot really expect a bijection, since the eigenspace is an
analytical object that depends sensitively on the edge weights,
whereas the group of mixed cycles is topological and doesn't depend on
the edge weights. Nevertheless, there is a sort of topological
stand-in for the eigenspace which nicely characterizes the modes which
do not cross over.  First we make a definition.

\begin{define}
  Let $S$ be a subspace (chosen independently of the weights
  $\gamma_{ij}$) and $P_S$ be the orthogonal projection onto the
  subspace $S$. The subspace $S$ is said to be a subspace of fixed
  index if the matrix
\[
 P_S {\mathcal L} : S \mapsto S
\]
has the same index regardless of the choice of weights.  A subspace of
fixed index is {\bf maximal} if $$\dim(S)=c(\Gp)+c(\Gm)-1.$$
\end{define}

It is a simple consequence of the Courant minimax principle that the 
projection of a symmetric matrix onto a subspace cannot have more 
positive or more negative eigenvalues than the original operator.
We know from theorem~\eqref{thm:main} that for appropriate choices of 
the weights the matrix can have as few as $c(\Gp)-1$ negative eigenvalues 
and (for a different choice of weights) $c(\Gm)-1$ positive eigenvalues. It 
thus follows that a subspace of fixed index cannot have more than $c(\Gp)-1$ 
negative eigenvalues, $c(\Gm)-1$ positive eigenvalues and one zero eigenvalue, 
and that the maximum possible dimension of a fixed subspace is $c(\Gp)+c(\Gm)-1.$ 
However it is not clear that one can actually have a maximal subspace of 
fixed index or, for that matter, any non-trivial subspace of fixed index at 
all. We conclude this section by showing that there is always a  maximal 
subspace of fixed index that has a natural topological construction. In 
essence this subspace gives a natural (orthogonal) decomposition into 
modes which do not have an eigenvalue crossing, and modes which do. 

First we need to define two complementary subspaces, one of which will 
be a maximal subspace of fixed index. The construction of these subspaces is
similar in spirit to the construction of the cut-space and cycle-space from 
algebraic topology, although the cut-space and cycle-space are subspaces 
of the vector space over the edge set, not the vector space over the vertices. 
For a nice description of the cut- and cycle-space and some applications to the 
theory of electrical networks see the paper of Bryant~\cite{Bryant.67}.

\begin{define}
Define $S_{\rm{free}}$ to be the following subspace of ${\mathbb R}^N:$ Given a 
basis for the mixed cycles $\{\gamma_i\}_{i=1}^{\tau}$ let $S_{{\rm free}}=\spanop{\{ v^{\gamma_i}\}_{i=1}^{\tau}.}$ 

Define $S_{\rm fixed}$ as follows: let $\Gp_i$ be the $i^{th}$
component of $\Gamma_+$, and let $\vec v^{i,+}$ be the characteristic
vector of $\Gp_i$:

\begin{equation*}
  v^{i,+}_j = \begin{cases}
    1, & j\in\Gamma_{+,i}, \\
    0, & j\not\in\Gamma_{+,i}.
    \end{cases}
\end{equation*}
Similarly, let $\vec v^{i,-}$ be the characteristic vector of
$\Gm_i$. Then
\[
S_{\rm fixed} = {\rm span}(\{v^{i,+}\}_{i=1}^{c(\Gamma_+)}, \{v^{i,-}\}_{i=1}^{c(\Gamma_-)}).
\]
\end{define}

\begin{lem}
The subspaces $S_{\rm free}$ and $S_{\rm fixed}$ are orthogonal complements:
\[
S_{\rm fixed} = S_{\rm free}^\perp
\]

\end{lem}  

\begin{proof}
First we check that $\dim(S_{\rm fixed})=c(\Gamma_+)+c(\Gamma_-)-1$. Note that the 
set of vectors $\{\vec v^{i,+}\}$ is linearly independent, as is the set of vectors $\{\vec v^{i,-}\}$. However  $\{\vec v^{i,+}\}\cup\{\vec v^{i,-}\} $ is 
not a linearly independent set, as one has 
\[
\sum_{i=1}^{c(\Gp)} \vec v^{i,+} = \sum_{i=1}^{c(\Gm)} \vec v^{i,-} = (1,1,1,\ldots,1). 
\]
We claim that this is the only relation. To see this note that we 
have the following identity for subspaces $Q,R$
\[
\dim({\rm span}(Q \cup R)) = \dim(Q) +  \dim(R) - \dim({\rm span}(Q\cap R)).
\]
Next note that ${\rm span}(\{\vec v^{i,+}\})\cap{\rm span}(\{\vec v^{i,-}\})$ consists of 
all vectors that are constant on components of $\Gamma_+$ and constant on components of 
$\Gamma_-$. Since $\Gamma$ is connected this means that these vectors 
must be constant on all of $\Gamma$, and are thus proportional to $(1,1,1,\ldots,1).$ 
It is clear that a basis for $S_{\rm free}$ is given by {\em any} $c(\Gamma_+)+c(\Gamma_-)-1$ 
vectors from $\{\vec v^{i,+}\} \cup \{\vec v^{i,-}\}.$ 

Next note that if $\gamma$ is a mixed cycle, with 
$\vec v^\gamma$ the corresponding vector, and $\vec v$ is constant on a 
component $\Gamma_{+,i}$, then 
\[
\langle\vec v^\gamma,\vec v\rangle=0.
\]
To see this first note that the number of times the cycle $\gamma$ enters 
 $\Gamma_{+,i}$ must equal the number of times it leaves  $\Gamma_{+,i}.$ When 
it enters and leaves  $\Gamma_{+,i}$ it must do so through a negative edge. 
Each time it enters  $\Gamma_{+,i}$ gives a $+1$ in some entry of $\vec v^\gamma$, 
and each time it leaves gives a $-1$ entry. Thus $\langle\vec v^\gamma,\vec v\rangle$ is 
the sum of an equal number of $+1$ and $-1$ entries and is therefore zero. 

Since $S_{\rm fixed}$ and $S_{\rm free}$ are orthogonal and have complementary dimensions they 
are orthogonal complements of one another.
\end{proof}

\begin{define}
 We define the following subspaces of $S_{\rm fixed}$
\begin{itemize}
\item $S_{\rm fixed}^+ = \{ w | w \in {\rm span}\{v^{i+}\}_{i=1}^{c(\Gp)} ~~{\rm and}~~ \langle w,(1,1,\ldots,1) \rangle=0\}$ 
\item $S_{\rm fixed}^- = \{ w | w \in {\rm span}\{v^{i-}\}_{i=1}^{c(\Gp)} ~~{\rm and}~~ \langle w,(1,1,\ldots,1) \rangle=0\}$ 
\item $S_{\rm fixed}^0 = {\rm span}((1,1,1,\ldots,1))$
\end{itemize}

\end{define}

\begin{lem}
The subspaces $S_{\rm fixed}^{+/-/0}$ are ${\mathcal L}$--orthogonal.
Specifically, what we mean by this is if $v$ and $w$ are chosen from
two different subspaces of $S_{\rm fixed}^{+}$, $S_{\rm fixed}^{-}$,
$S_{\rm fixed}^{0}$, then
\[
\langle v, {\mathcal L} w \rangle=0.
\]

\label{lem:block}
\end{lem}

\begin{proof}
$S_{\rm fixed}^0 \subset \ker({\mathcal L})$, so the fact that any
  inner product involving that subspace is obvious.  Thus we restrict
  our attention to $S_{\rm fixed}^{\pm}$.  By definition $\vec
  v^{i^\prime,-}$ is constant on the component $\Gamma_{i^\prime,-}$
  and zero off of this component.  Let $\partial\Gamma_{i^\prime,-}$
  denote the set of vertices which are not in $\Gamma_{i^\prime,-}$
  but which are connected to it by a positive edge: the nearest
  neighbors of the component.  By direct computation it is easy to see
  that ${\mathcal L}\vec v^{i^\prime,-} $ takes the following form:
\[
({\mathcal L}\vec v^{i^\prime,-})_j = \left\{
\begin{array}{c} -\sum_{k \in \partial\Gamma_{i^\prime,-}} \gamma_{k,j} \qquad j \in \Gamma_{i^\prime,-} \\ 
\sum_{k \in  \Gamma_{i^\prime,-}}  \gamma_{k,j} \qquad j \in \ \partial\Gamma_{i^\prime,-}.
\end{array}\right. 
\]
Thus ${\mathcal L}\vec v^{i^\prime,-}$ is a linear combination of vectors that are $1$ on some vertex in $\Gamma_{i^\prime,-}$ and 
$-1$ on some vertex not in $\Gamma_{i^\prime,-}$ but connected to it by a positive edge. Each of these vectors is necessarily 
orthogonal to $\vec v^{i,+}$ since $\vec v^{i,+}$ is constant on components of $\Gamma_{+}$. 
\end{proof}

We need one more lemma to prove the final result, the well-known Sylvester theorem
\begin{lem}[Sylvester's law of Inertia]
If $A$ is a square matrix and $S$ a square non-singular matrix then $A$ and $B =S^\dagger A S$ 
have the same index.  The matrices $A$ and $B$ are said to be Sylvester equivalent.
\end{lem}
\begin{proof}
  This is an old result and many proofs are known. We include a short
  one here for completeness. Consider the one-parameter family of
  Hermitian matrices $A(s) = (\cos(s) I - i \sin(s) B^\dagger) A
  (\cos(s) I + i \sin(s) B^\dagger).$ The matrices $(\cos(s) I - i
  \sin(s) B^\dagger)$ and $(\cos(s) I + i \sin(s) B^\dagger)$ are both
  invertible, so $\dim(\ker(A(s))$ is independent of $s$, and there
  are no eigenvalue crossings. This gives a homotopy from $A(0)=A$ to
  $A(\frac\pi{2})=B^\dagger A B$ that does not change the index, so
  the indices of the two must be equal.
\end{proof}

\begin{prop}
The subspace $S_{\rm fixed}$ is a maximal subspace of fixed index: The 
operator $P_{S_{\rm fixed}} {\mathcal L}$  (considered as 
an operator from $S_{\rm fixed}$ to $S_{\rm fixed}$) has index
\[
\ind(P_{S_{\rm fixed}} {\mathcal L}) = (c(\Gamma_+)-1,1,c(\Gamma_-)-1)
\] 
independent of the choice of edge weights. 
\end{prop}
\begin{proof}
 Using Lemma \ref{lem:block} it follows that $P_{S_{\rm fixed}}
 {\mathcal L}$ has the following block structure:
\[
 P_{S_{\rm fixed}} {\mathcal L}P_{S_{\rm fixed}} = {\mathcal L}_- \oplus {\mathcal L}_0 \oplus  {\mathcal L}_+
\]
where ${\mathcal L}_+= P_{S_{\rm fixed}^+} {\mathcal L}$, ${\mathcal
  L}_-= P_{S_{\rm fixed}^-} {\mathcal L}$ and ${\mathcal L}_0=
P_{S_{\rm fixed}^0} {\mathcal L}=0.$

 Consider the block ${\mathcal L}_- \oplus {\mathcal L}_0 $. This
 arises by orthogonal projection of ${\mathcal L}$ onto the set of
 vectors constant on $\Gm$, and thus this matrix is Sylvester
 equivalent to the following graph Laplacian: one contracts on the
 negative edges, giving a graph with vertices corresponding to the
 components of $\Gm$.  This is a standard graph Laplacian on a
 connected graph with $c(\Gm)$ vertices and thus has $c(\Gm)-1$
 negative eigenvalues and one zero eigenvalue. Similarly $ {\mathcal
   L}_0 \oplus {\mathcal L}_+$ is Sylvester equivalent to the negative
 of the graph Laplacian given by contracting on the positive
 edges. This has, by the same argument, $c(\Gp)-1$ positive
 eigenvalues and one zero eigenvalue. The zero eigenvalue is obviously
 counted twice in this argument, giving
\[
{\rm ind}\left( P_{S_{\rm fixed}} {\mathcal L}P_{S_{\rm fixed}}\right) = (c(\Gm)-1,1,c(\Gp)-1).
\]

\end{proof}

\section{Refinements}\label{sec:refinements}

In this section, we present some refinements of
Theorem~\ref{thm:main}. In Section~\ref{sec:dct} we present the
Deletion-Contraction Theorem, which allows us to obtain recursive
formulas for the $\M(\G(t))$ polynomial in terms of ``smaller'' graphs
obtained by deleting and contracting edges; this allows us to give a
precise characterization of the bifurcation structure of $\L(\G(t))$
in many cases.  In Section~\ref{sec:spectrum}, we discuss the
asymptotics of the individual eigenvalues of $\L(\G(t))$ in the limits
$t\to0,\infty$.  Recall that in the statement and proof of
Theorem~\ref{thm:main}, we discuss the number of eigenvalues in the
left- and right-hand half-planes in these two limits; here we give
more precise statements of the locations of these eigenvalues.

\newcommand{\del}[2]{{#1}_{/{#2}}}
\newcommand{\con}[2]{{#1}_{.{#2}}}
\newcommand{\concon}[3]{{#1}_{.{#2}.{#3}}}
\newcommand{\condel}[3]{{#1}_{.{#2}\setminus{#3}}}
\newcommand{\deldel}[3]{{#1}_{\setminus{#2}\setminus{#3}}}
\newcommand{\delcon}[3]{{#1}_{\setminus{#2}.{#3}}}

\subsection{ Deletion--contraction theorem}\label{sec:dct}

\begin{define}\label{def:special}
  Let $\G = (V,E)$ be a weighted multigraph (loops and multiple edges
  allowed), and $e\in E(\G)$ an edge.  Let $\del\G e$ denote the graph
  obtained by removing edge $e$: the graph with vertex set $V(\G)$ and
  edge set $E(\G)-e$.  If $e$ is an edge which is not a loop let
  $\con\G e$ denote the graph obtained by contracting on the edge,
  which is obtained as follows:
\begin{itemize}
 \item If edge $e$ connects vertices $v_1$ and $v_2$ then $v_1$ and $v_2$ are identified as a single vertex 
$v^*$.
\item Each edge connecting a vertex to $v_1$ or $v_2$ becomes an edge connecting that vertex to $v^*$ 
\item Edge $e$ is removed from the edge set. 
\end{itemize}

Note that if $\G$ is a simple graph and $v_1$ and $v_2$ are part of a triangle then the contracted 
graph  $\con\G e$ will have multiple edges. Similarly if there are multiple edges connecting 
$v_1$ and $v_2$ then  $\con\G e$ will have loops. 
\end{define}

We can now state the Deletion-Contraction Theorem.

\begin{thm}[Deletion-Contraction theorem]\label{thm:dct}
  Let $\G$ be a weighted multigraph with $e\in E(\G)$.  Then $\M(\G)$ can be computed by applying the 
following rules
\begin{itemize}
\item If $e$ is not a loop then $\label{eq:dct} \M(\G) = \M(\del\G e) + \gamma_e \M(\con\G e).$
\item If $e$ is a loop then $\M(\G) = \M(\del\G e).$
\item If $\G$ is disconnected then  $\M(\G) =0$
\end{itemize}
\end{thm}

\begin{proof}
  The deletion-contraction recursion is well-known --- see Chapter
  13.2 of the text of Godsil and Royle~\cite{godsil.royle} for one
  proof.
\end{proof}

\begin{figure}[ht]
\begin{centering}
  \includegraphics[height=1in]{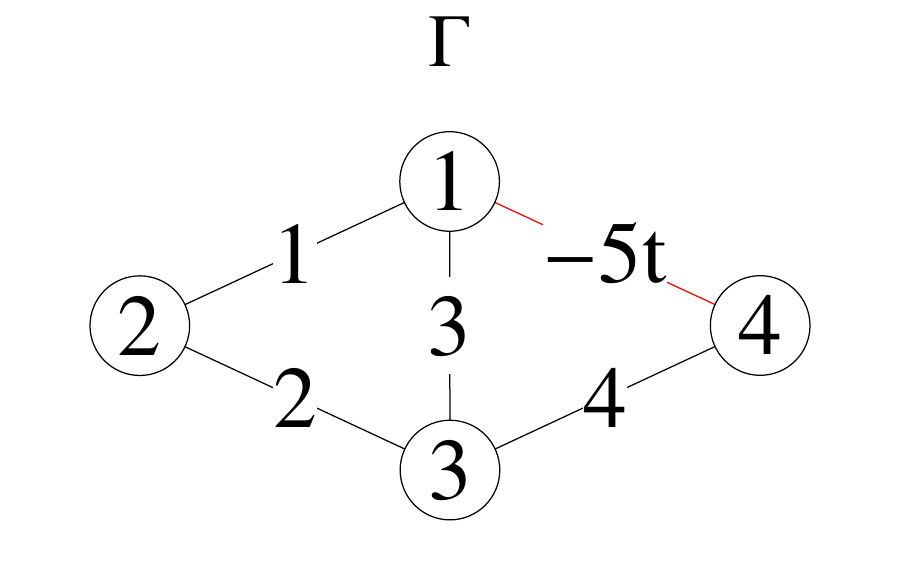}%
  \includegraphics[height=1in]{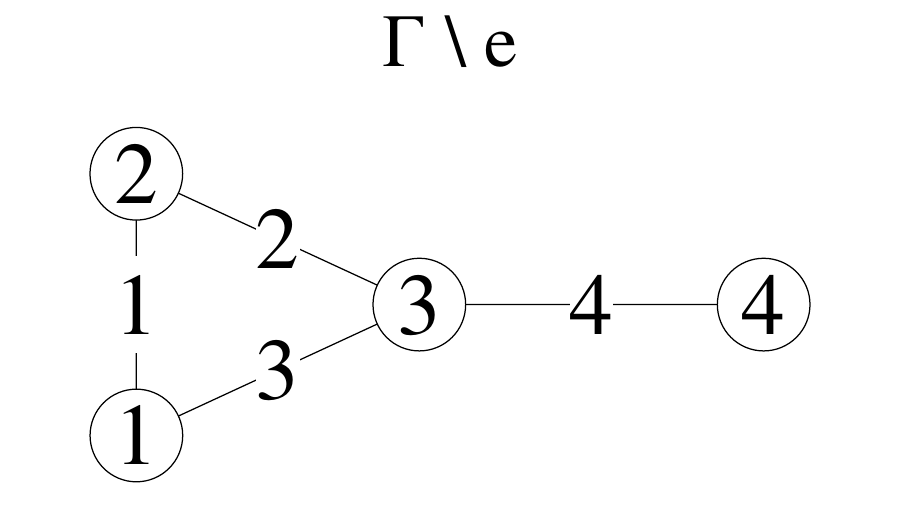}%
  \includegraphics[height=1in]{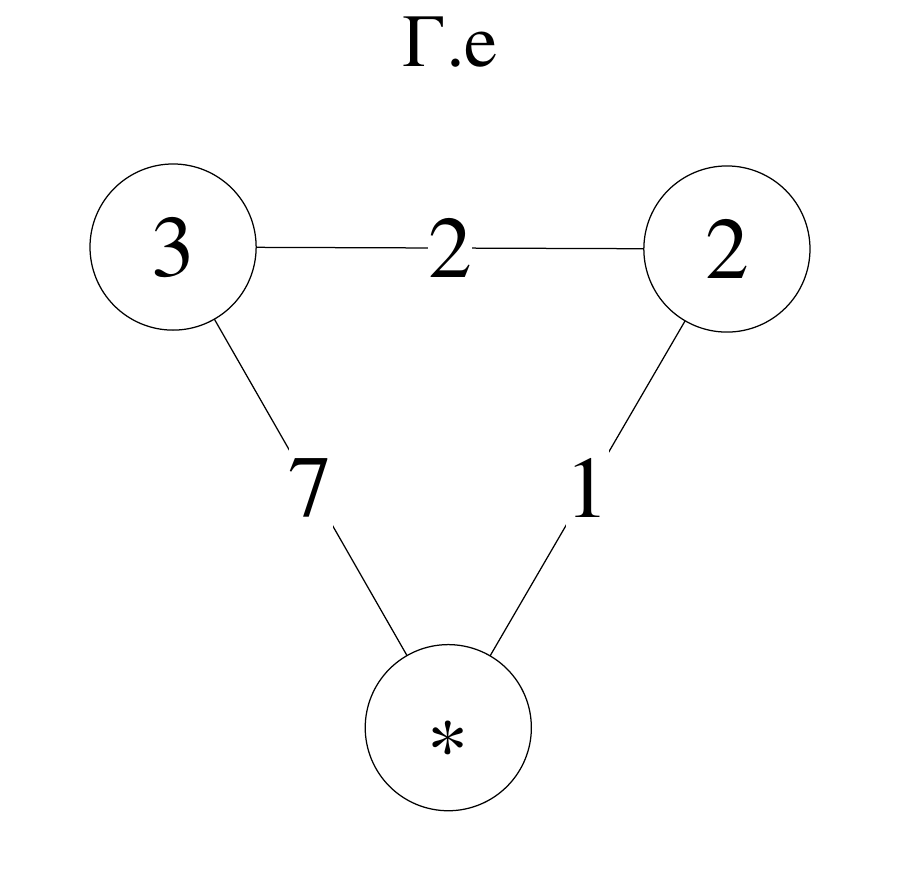}
  \caption{The graphs $\G; \G\setminus e; \G.e$, where $e$ is chosen
    to be the edge with negative weight.}
  \label{fig:dct}
\end{centering}
\end{figure}

\begin{example}
  As an example, let us consider the graphs given in
  Figure~\ref{fig:dct} (here, think of the $t$ as a symbol to separate
  out the terms containing the special edge).  First, notice that $\G$
  has three spanning trees not containing $e$, of weights $8,12,24$,
  and five spanning trees containing $e$, of weights $-10t, -15t,
  -20t, -30t, -40t$.  By Lemma~\ref{lem:mtt},
  \begin{equation*}
    \M(\G) = 44 - 115 t.
  \end{equation*}
  Notice that every spanning tree of $\G$ that does not contain $e$ is
  also a spanning tree of $\G\setminus e$, and thus
  \begin{equation*}
    \M(\G\setminus e) = 44.
  \end{equation*}
  On the other hand, $\Gamma.e$ has three spanning trees of weights
  $2,7,14$.  Therefore we have
  \begin{equation*}
    \M(\G.e) = 23.
  \end{equation*}
  And we see that 
  \begin{equation*}
    \M(\G) = 44-115t = 44 + (-5t) (23)= \M(\del\G e) - \g_e\M(\con\G e).
  \end{equation*}
\end{example}


In the context of dynamical systems, an important distinction to be
made is that between graphs for which $n_+(\L(\G)) = 0$ and those for
which $n_+(\L(\G)) > 0$; the former are called {\em stable} and the
latter {\em unstable}.  This notion comes from the fact that if we
consider the ordinary differential equation
\begin{equation*}
  \dot x = \L(G)x, 
\end{equation*}
then the origin is stable to perturbations iff $n_+(\L(G)) = 0$; if
not, then perturbations move away from the origin at an exponential
rate.

Moreover, as proved above, when we consider the homotopy $\G(t) = \G_+
+ t\G_-$, the function $n_+(\L(\G(t)))$ is a non-decreasing function
of $t$.  From this, it follows that if we define
\begin{equation*}
  \tstar(\G) := \sup_{t\ge 0} n_+(\L(\G(t))) = 0,
\end{equation*}
then for $t \le \tstar(\G)$, the Laplacian is stable, and for
$t>\tstar(\G)$, the Laplacian is unstable---in short, it undergoes a
dynamical {\em bifurcation}.

Two facts follow immediately from Theorem~\ref{thm:main}: first, that
$\tstar(\G) > 0$ if and only if $\G_+$ is connected, and $t^*(\G) <
\infty$ if and only if $\G_- \neq 0$.  We will concentrate the most in
what follows on the case of $\tstar(\G) \in (0,\infty)$.

We also point out that rigid graphs undergo no bifurcation; in fact
$n_+(\G(t))$ is constant on $\{t>0\}$.  From the bifurcation point of
view, then, these are the least interesting cases.


We now work out some special cases.

\subsubsection{One negative edge}

Let us first consider the case where $\Gp$ is connected, and there is
exactly one negative edge in $\Gm$, call it $e$.  Then we have
\begin{equation*}
  \M(\G(t)) = \M(\Gp) - t \av{\g_e} \M(\G.e).
\end{equation*}
This means that
\begin{equation*}
  \tstar = \frac{\M(\Gp)}{\av{\g_e} \M(\G.e)}.
\end{equation*}
Notice that this formula gives $t^\star>0$; since $\Gp$ and $\G.e$
have non-negative entries, they will have opposite signs, since their
dimensions differ by one~(q.v.~\eqref{eq:defofM}).  

%
%
%
%
%
%
%



Intuitively, we expect that the denser a graph is, the less powerful
each individual edge would be.  We present a couple of examples.

\begin{example}[Ring graph]
  Consider the case of the ring graph $R_N$, i.e. $V(R_N) = \{1,\dots,N\}$ and 
\begin{equation*}
  (R_N)_{ij} = \begin{cases} 1,& j = i\pm1 \pmod N,\\ 0,&\mbox{else.}\end{cases}
\end{equation*}
We also define the path graph $P_N = R_N \setminus \{1,N\}$.  Choose
any edge of $R_N$ and flip its sign to minus one.  According to the
Proposition above, the homotopy $\L(\G(t))$ will lose stability at
\begin{equation*}
  \tstar = \frac{\M(\Gamma\setminus e)}{\M(\Gamma.e)} = \frac{\M(P_N)}{\M(R_{N-1})}.
\end{equation*}
It is easy to see that $\M(R_N) = N, \M(P_N) = 1$, so we have
\begin{equation*}
  t^* = \frac{1}{N-1}.
\end{equation*}
\end{example}

\begin{example}[Complete graph] 
  Consider the complete graph $K_N$.  Again choose any edge, then the
  deletion operators give graphs $OG_N, DR_{N-1}$.  $OG_N$ is the complete
  graph minus one edge; $DR_{N}$ is a complete graph on $N$ nodes
  with all of the edges going to one distinguished vertex having twice
  the weight.

  Using standard counting arguments, we see that
  \begin{equation*}
    \M(OG_N) = (N-2) N^{N-3},\quad \M(DR_N) = 2(N+1)^{N-2},
  \end{equation*}
  giving
  \begin{equation*}
    \tstar = \frac{\M(OG_N)}{\M(DR_{N-1})} = \frac{(N-2) N^{N-3}}{2 N^{N-3}} = \frac{N-2}2.
  \end{equation*}
  Of course, this can be obtained by other means; in fact, one can
  compute that the eigenvalues of $\G(t)$ are
  \begin{equation*}
    \{0\}\cup \{-N\}^{(N-2)}\cup \{2 t + (N-2)\},
  \end{equation*}
  and the computation would also follow from this.
\end{example}

\begin{conjecture}
  These are the extreme cases; for any graph $\Gamma$ with $\Gp$
  connected and $\av{E(\Gm)} = 1$, we have $$\tstar(\Gamma) \in [(N-1)^{-1},
    (N-2)/2].$$
\end{conjecture}

\begin{prop}
  Assume that $\av{E(\Gm)} = 2$ and that the two edges $e,f\in E(\Gm)$
  do not share a vertex.  Then
  \begin{equation*}
    \M(\G(t)) = t^2 \av{\g_e\g_f}\M(\concon \G e f) - t\l(\av{\g_e}\M(\condel \G e f) + \av{\g_f}\M(\delcon\G e f)\r) + \M(\deldel \G e f),
  \end{equation*}
  where we have defined all of these terms in
  Definition~\ref{def:special}.  Also, $t^*(\Gamma)$ is the minimal
  positive root of this polynomial.
\end{prop}

\begin{proof}
  Let $\Gamma$ be any signed graph; from Theorem~\ref{thm:dct}, we
  have 
  \begin{equation*}
    \M(\G) = \g_e \M(\con\G e) + \M(\del \G e),
  \end{equation*}
  and
  \begin{equation*}
    \M(\con\G e)  = \g_f \M(\con{(\con\G e)}f) + \M(\del{(\con\G e)}f),    \quad \M(\del\G e)  = \g_f \M(\con{(\del\G e)}f) + \M(\del{(\del\G e)}f),
  \end{equation*}
  and iterating~\eqref{eq:dct} gives
  \begin{equation*}
    \M(\G) = \g_e\g_f \concon\G e f + \g_e \M(\condel\G f e) + \g_f \M(\condel \G e f) + \M(\deldel\G e f).
  \end{equation*}
  If the edges $e$ and $f$ are in $E(\G_-)$, then $\g_{e} = -t
  \av{\g_{e}}, \g_f = -t\av{\g_f}$, and the result follows.
\end{proof}

\begin{remark}
  Note that this polynomial has positive roots by
  Theorem~\ref{thm:main}.  Consider, for example, an unweighted graph;
  this implies~(q.v.~Remark~\ref{rem:logconcave}) that
  \begin{equation*}
    \l(\M(\Gamma.e\setminus f) + \M(\Gamma\setminus e.f)\r)^2 > 4 \M(\Gamma.e.f)\M(\Gamma\setminus e\setminus f).
  \end{equation*}
\end{remark}

%
%
%
%
%
%
\subsection{Detailed Eigenvalue Asymptotics}\label{sec:spectrum}

We now consider more detailed asymptotics of the eigenvalue spectrum
in the limits in which the strength of the negative edges is much
weaker or much stronger than the strength of the positive
edges. Specifically we consider the one-parameter family of graph
Laplacians
\begin{equation*}
  \L(\G(t)) = \L(\Gp) - t \L(\Gm),
\end{equation*}
 in the limits $ t \rightarrow 0^+$ and $t \rightarrow \infty$.  The
 spectrum splits naturally into two parts, which correspond to the
 eigenvalues of graph Laplacians on the deleted and contracted
 graphs. This is an analog on the level of the spectrum of the
 contraction-deletion algorithm for computing the crossing polynomial:
 the crossing polynomial is given by the sum of the crossing
 polynomials for the contracted and deleted graphs, while the spectrum
 is given (asymptotically!) by the union of the deleted and contracted
 graphs.

\begin{thm}
Suppose that $t$ is large and positive. Then ${\mathcal L}(t)$ has exactly
$N-c(\Gm)$ negative eigenvalues, $c(\Gm)-1$ positive eigenvalues and one 
zero eigenvalue. If we take the convention that eigenvalues are numbered
in decreasing order then to leading order in $t$ the  $c(\Gm)-1$ positive 
eigenvalues are given by 
\[
\lambda_i(\L_\Gamma) = t \lambda_i(\L_{\Gamma_-}) + O(1) \qquad\qquad i\in \{1\ldots c(\Gm)-1\},  
\]
The $N-c(\Gm)$ negative eigenvalues are given to leading order by
\[
\lambda_i(\L_\Gamma) = \tilde\lambda_i(\L_{\Gamma\cdot-}) + o(1) \qquad\qquad i\in \{N\ldots c(\Gm)+1\},
\]
where $(L_{\Gamma\cdot-}$ is the graph formed by contracting on the negative 
edges. Here $ \tilde \lambda_j(\L_{\Gamma\cdot-})$ are solutions to 
\[
\L_{\Gamma^-} \vec v = \lambda S \vec v 
\]
where $L_{\Gamma^-}$ is the graph Laplacian formed by contracting on
the negative edges and $S$ the contracted inner product: the diagonal matrix 
with entries $S_{ii} =|V(\Gamma_{-,i})|$.
\end{thm}

\begin{proof}
The proof follows in a straightforward way from perturbation theory for 
eigenvalues of a symmetric matrix. The full graph Laplacian can be written 
\[
{\mathcal L}(t) = t \left({\mathcal L}(\Gamma_-) + t^{-1}{\mathcal L}(\Gamma_+)\right)
\]
so it suffices to understand the eigenvalues of ${\mathcal
  L}(\Gamma_-) + t^{-1}{\mathcal L}(\Gamma_+)$ for $t$ large. The
spectrum of ${\mathcal L}(\Gamma_-)$ consists of $c(\Gamma_-)$ zero
eigenvalues and $N-c(\Gamma_-)$ positive eigenvalues. For the non-zero
eigenvalues straightforward eigenvalue perturbation theory gives the
asymptotic above.

To understand how the $c(\Gamma_-)$-dimensional kernel breaks under
perturbation we must do a degenerate perturbation theory
calculation. Well-known results (again see Kato) show that to leading
order the eigenvalues are given by the eigenvalues of the reduced
matrix
\[
P_{\ker({\mathcal L}_-)} L_{\Gamma_+} P_{\ker({\mathcal L}_-)} 
\] 
where $P_{\ker({\mathcal L}_-)}$ is the orthogonal projection onto the
kernel of ${\mathcal L}_-$. It is straightforward to compute
$\ker({\mathcal L}_-)$: it consists of vectors that are {\em constant
  on components of ${\mathcal L}_-$.} We can identify a vector $\vec
w\in \ker({\mathcal L}_-)$ with a vector $\tilde w \in {\mathbb
  R}^{c(\Gamma_-)}$ by the following rule: if $w_i = \alpha$ for all
vertices in component $\Gamma_-,j$ then $\tilde w_j=\alpha$. Under
this identification the natural inner product on ${\mathbb R}^N$ maps
to the inner product
\[
\langle\tilde v, \tilde w\rangle = \sum_{i=1}^{c(\Gamma_-)} |V(\Gamma_{i,-})|\tilde v_i \tilde w_j.
\]  
In other words the there is one entry per component of $\Gamma_-$, and the 
inner product is diagonal with weights given by the number of vertices 
in the corresponding component of $\Gamma_i.$  It is straightforward to 
see that the matrix 
\[
P_{\ker({\mathcal L}_-)} L_{\Gamma_+} P_{\ker({\mathcal L}_-)} 
\] 
is exactly the Laplace matrix obtained by contracting on the negative edges 
of the graph, completing the proof. 
\end{proof}

\newcommand{\specstar}{{\mathrm {Spec}}^\star}

\begin{define}
  Given a matrix $A$, we define $\specstar(A)$ as the eigenvalues of
  the matrix $A$ restricted to the subspace $(1,1,1,\ldots,1)^\perp$
  of mean zero vectors.
\end{define}

\begin{remark}
 The previous theorem can be written in the following compact way
\[
 \specstar({\mathcal L}(t)) \approx \left\{\begin{array}{c} \
     \specstar(\Gm) \cup \specstar(\G_{\cdot -})\qquad t \rightarrow +\infty \\
     \specstar(\Gp) \cup \specstar(\G_{\cdot +})\qquad t\rightarrow  0^+             
                                          
\end{array}\right.
\]
with the understanding that the eigenvalues for the contracted graph
are taken with respect to the natural inner product $S$. This shows that 
there is an (approximate) contraction-deletion relation at the level of the 
spectrum analogous to the contraction-deletion relation satisfied by the crossing polynomial. 
\end{remark}

\begin{example}\label{exa:asymptotics}
We consider the following graph, where all positive (solid) edges are
weighted $+1$ and all negative (dashed) edges weighted $-1$:
\begin{figure}[ht]
\begin{centering}
\includegraphics[width=2in]{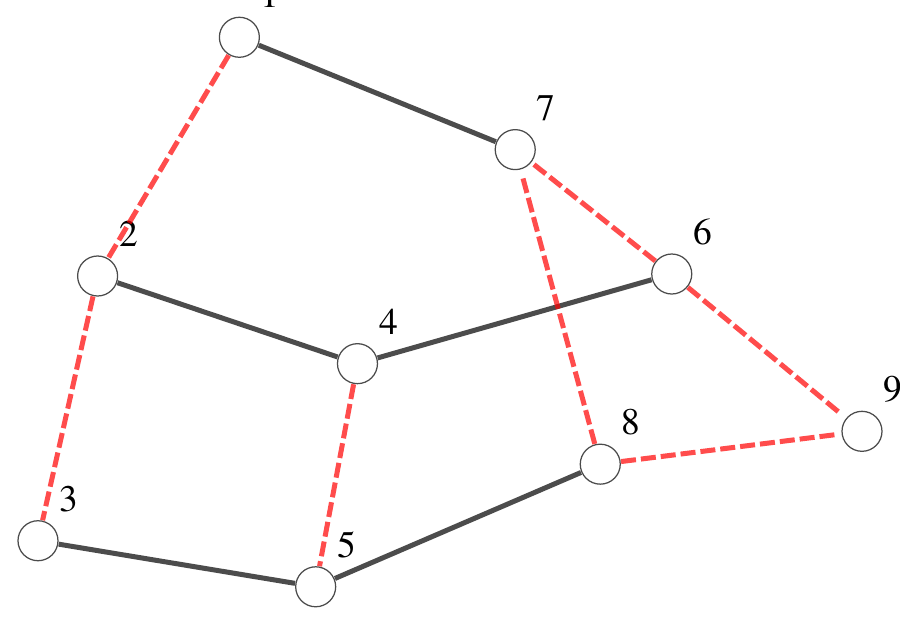}
\includegraphics[width=2in]{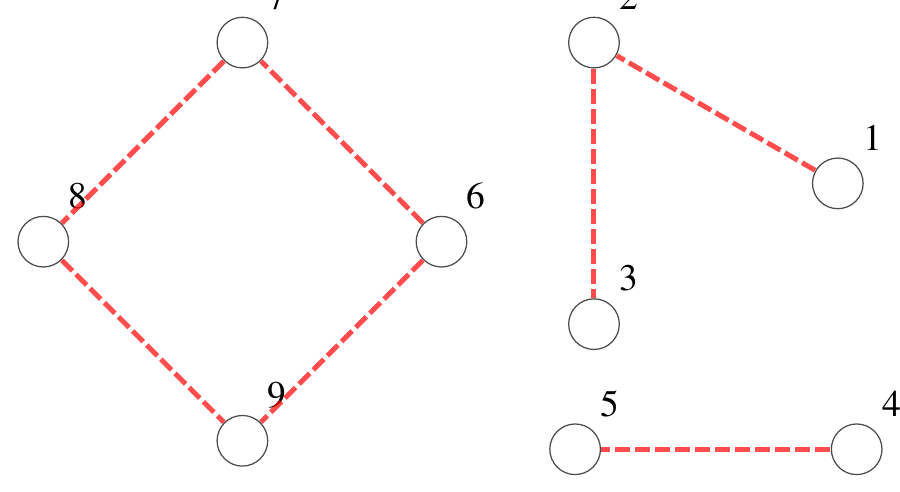} 
\includegraphics[width=1in]{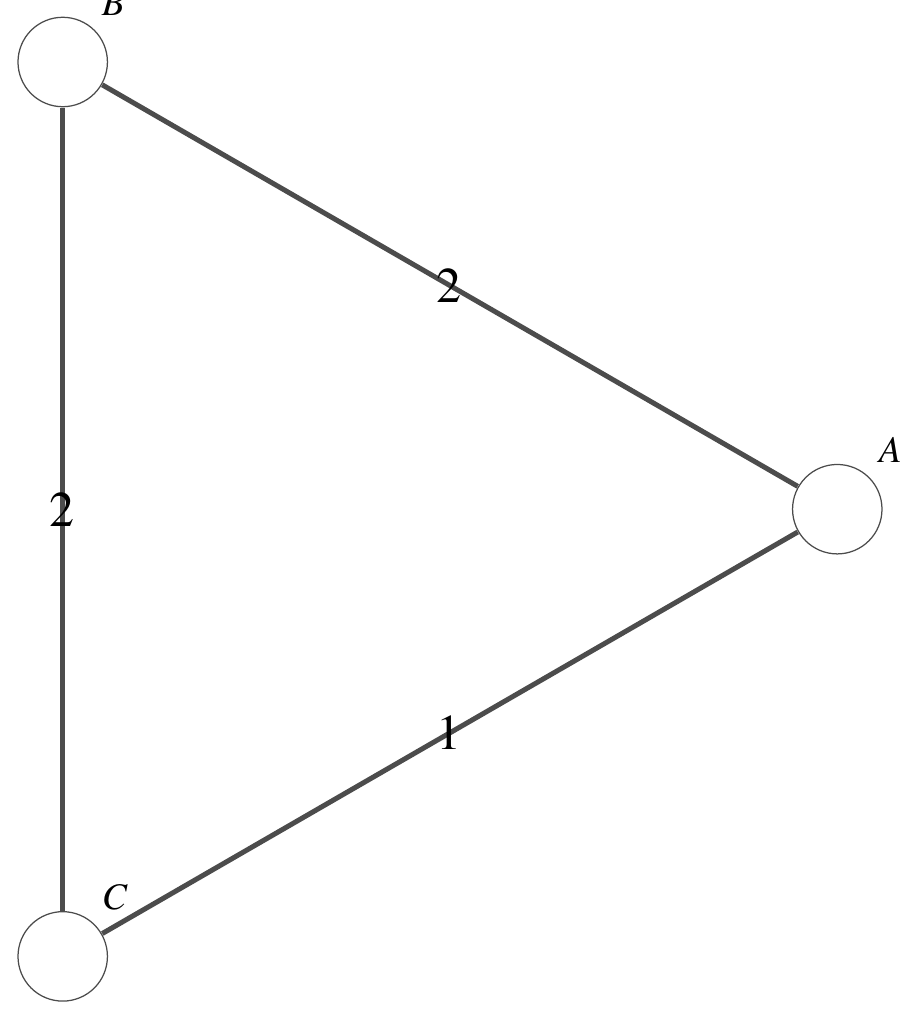}
\caption{The graph $\G$; the graph $\Gm$; the graph $\G_{\cdot-}$}
\label{fig:asymptoticGraphs}
\end{centering}
\end{figure}
The full graph has five positive edges, seven negative edges and nine
vertices. The subgraph $\Gamma_-$ consisting of only the negative
links has three components. Component A consists of vertices 1,2 and 3
and connecting edges , component B consists of vertices 4 and 5 and
the connecting edge, and Component C consists of vertices 6, 7, 8, and
9 and connecting edges. There are $N - c(\Gm)=9-3=6$ non-zero
eigenvalues corresponding to the graph Laplacian associated with the
negative edges. The non-zero eigenvalues associated to component A are
$1$ and $3$, to component B is $2$ and to component C are $4$, $2$ and
$2$. This gives six eigenvalues that grow linearly:
\begin{eqnarray*}
\lambda_9 \approx 4 t + O(1) \\
\lambda_8 \approx 3 t + O(1) \\
\lambda_7 \approx 2 t + O(1) \\
\lambda_6 \approx 2 t + O(1) \\
\lambda_5 \approx 2 t + O(1) \\
\lambda_4 \approx 1 t + O(1) \\
\end{eqnarray*}
If one contracts on all of the dashed edges the full graph reduces to
the three cycle, with one vertex corresponding to each component of
$\Gm$. There are two edges between components A and B, two between
components B and C, and one between A and C, so these edges are
weighted accordingly. The norm is contracted as well, and the norm
over the new vertex space can be written as
\[
\Vert \vec v \Vert^2 = \vec v^t S \vec v
\]
where $S$ is the matrix 
\[
S =  \left(\begin{array}{ccc}3 & 0 & 0 \\ 0 & 2 & 0 \\ 0 & 0 & 4\end{array}\right)
\]
The diagonal entries reflect the fact that the components have three, two and 
four vertices respectively. Thus the eigenvalue problem becomes  
\[
\left(\begin{array}{ccc}-3 & 2 & 1 \\ 2 & -4 & 2 \\ 1 & 2 & -3\end{array}\right)\vec v = \lambda \left(\begin{array}{ccc}3 & 0 & 0 \\ 0 & 2 & 0 \\ 0 & 0 & 4\end{array}\right)\vec v
\]
giving the negative eigenvalues as 
\begin{eqnarray*}
\lambda_2(t) \approx \frac18\left(\sqrt{33}-15\right) + O({1}/{t})\\
\lambda_1(t) \approx \frac18\left(-\sqrt{33}-15\right) + O({1}/{t})
\end{eqnarray*}
Finally the flexibility is equal to $\tau(\G) = 10 - 3 - 4 = 3$, so there are 
three eigenvalue crossings. It is straightforward though tedious to compute that 
the crossing polynomial is given by $P_\Gamma(t) = 171 t^3 - 702 t^4 + 828 t^5 - 288 t^6.$
The non-zero roots occur at $t \approx .43, t \approx .90, t \approx 1.55$.

 Some numerical results are shown in Figure {}. The first plot shows a graph of $\lambda_i(t)/t$ for $i = 3 \ldots 9$ and  $t\in (0,8).$ It seems clear 
that the scaled eigenvalues are converging to the correct values. The second 
plot shows a plot of the (unscaled) negative eigenvalues for  $t\in (0,8).$ Again it is clear that they are converging 
to $-\frac{1}{8}(\sqrt{33}+15)$ and  $-\frac{1}{8}(15-\sqrt{33})$ respectively. One can also see that there are three 
eigenvalue crossings at the correct $t$ values.  
\begin{figure}[ht]
\begin{centering}
\includegraphics[width=2in]{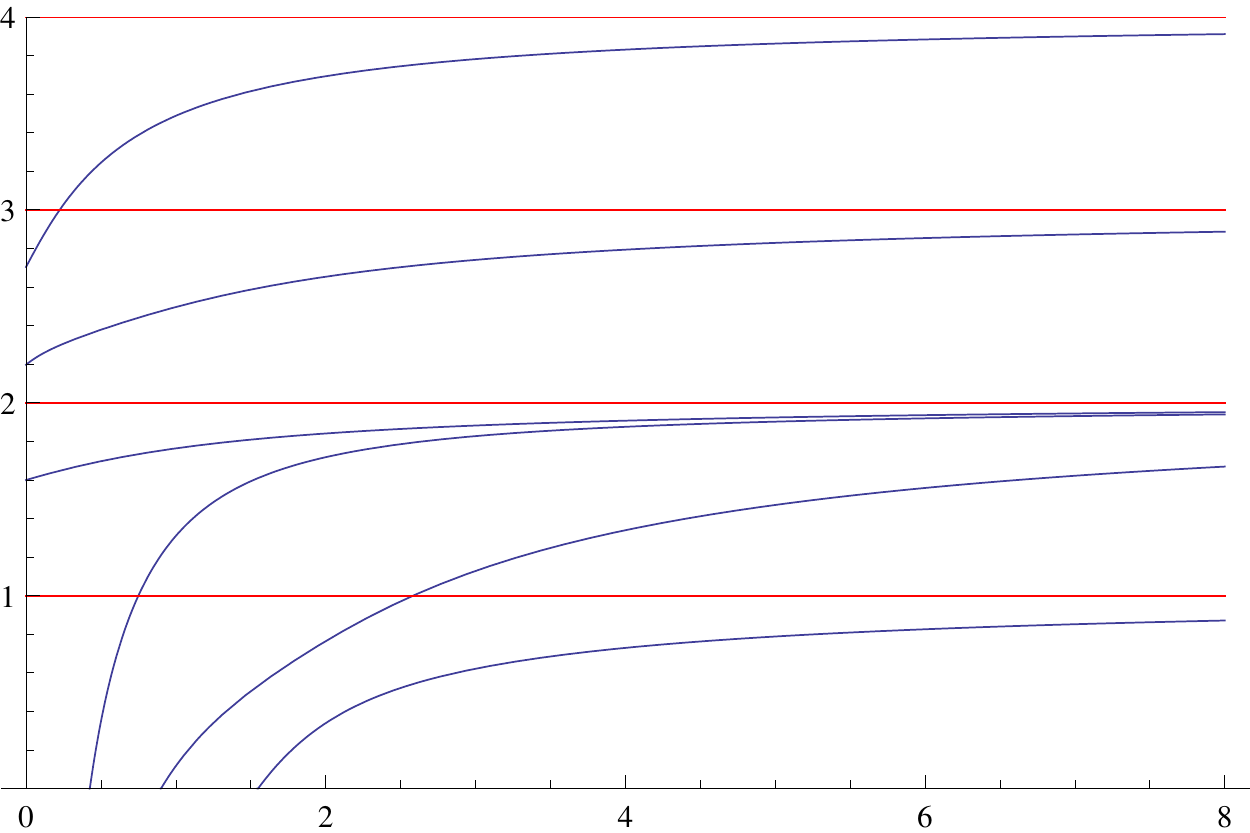} \includegraphics[width=2in]{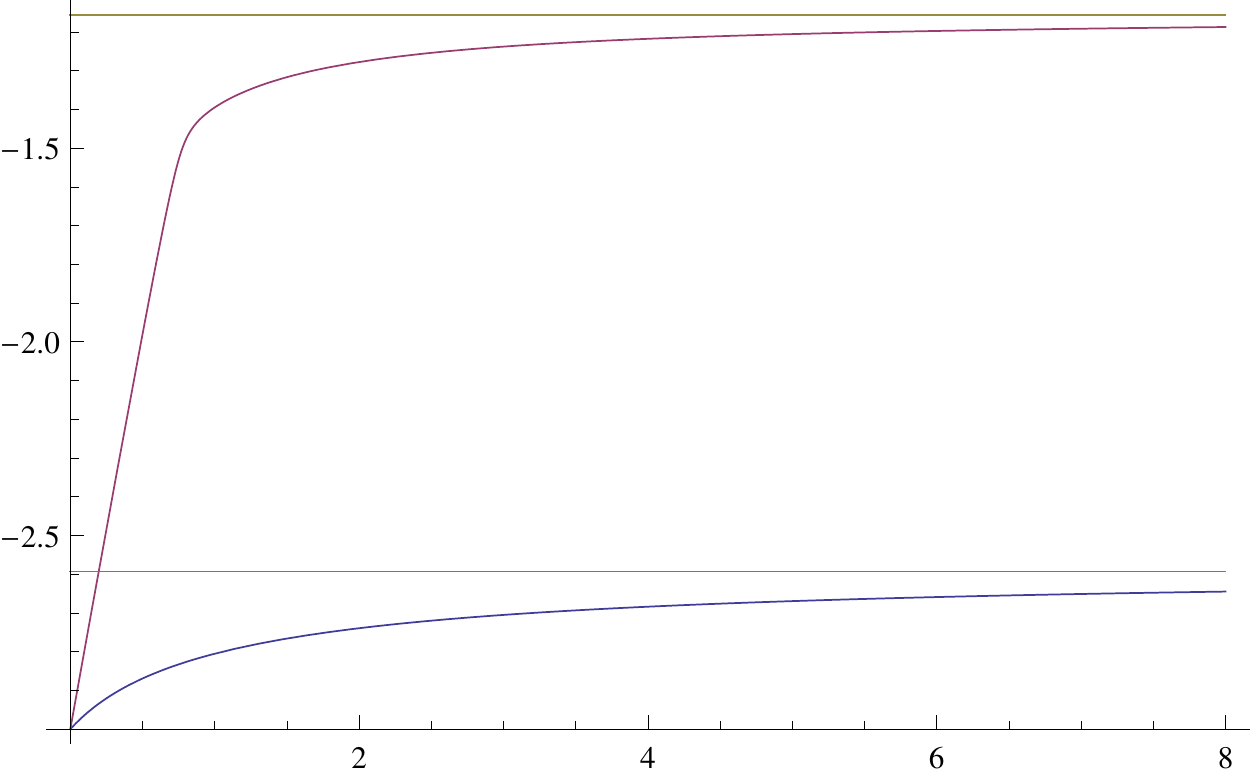} 
\caption{Eigenvalues of $\G(t)$ as a function of $t$, where $\G(t)$ is
  defined in Example~\ref{exa:asymptotics}.}
\label{fig:asymptotics}
\end{centering}
\end{figure}
\end{example}

\subsection{Comparison to the Gershgorin Theorem}\label{sec:Gershgorin}

In applied mathematics and numerical analysis it is often necessary to estimate 
the locations of the eigenvalues of a matrix or linear operator. One simple 
and very widely used tool for this is the Gershgorin 
disc theorem, which says the following: 
\begin{thm}[Gershgorin]
Given a matrix $M$ with entries $M_{ij}$ define the following $n$ closed disks:
\[
 D_i = \{ z | |z-M_{ii}|\leq \sum_{j \neq i} |M_{ij}|\}.
\]
 Then the eigenvalues of $M$ lie in the union of the disks, ${\rm spec}(M) \in \cup_{i=1}^n D_i.$
\end{thm}

The fact that the disks are defined using the edge weights might make
one think that the Gershgorin theorem might give more information on
the signs of the eigenvalues than the purely topological arguments,
but this is not, in fact, the case.  It is not difficult to see that,
for a non-trivial signed Laplacian the Gershgorin theorem always gives
results that are strictly worse than those given by Theorem
\ref{thm:main}.

First note that if the matrix $M$ represents a graph Laplacian then
there are three kinds of discs. If all of the edges emanating from the
vertex are all positive (resp. all negative) then the associated disc
lies in the closed left half-plane (resp. right half-plane) and is
tangent to the origin and the corresponding eigenvalue is non-positive
(resp. non-negative). If the vertex has edges of both signs then the
origin lies in the interior of the disk and the sign of the eigenvalue
cannot be determined. It is these discs that correspond to eigenvalues
whose sign cannot be determined. The main observation in this section
is that the number of such discs is always strictly larger than
$\tau.$

\begin{prop}
Suppose the graph $\G$ is connected and contains edges of both signs.
Let $n$ be the number of Gershgorin discs for which the origin lies in
the interior. Then $n \ge \tau+1$.
\end{prop}
\begin{proof}
 This follows immediately from the topological characterization of the
 number of of modes for which there is an eigenvalue crossing. The
 results of the previous section show that these can be associated
 with mixed cycles. By the construction given in that section we
 associate to each of these cycles a vector that has $+1$ for each
 time the cycle enters a vertex from a positive edge and leaves it by
 a negative edge, and $-1$ for each time the cycle enters a vertex
 from a negative edge and leaves it by a positive edge. These vectors
 can only have non-zero entries in vertices which have both types of
 edge, so they obviously lie in a subspace isomorphic to ${\mathbb
   R}^n$. These vectors are necessarily orthogonal to
 $(1,1,1,\ldots,1)$, so there can be at most $n-1$ linearly
 independent ones.
\end{proof}

\begin{remark}
 Note that $n=\tau+1$ can be achieved --- one example is when
 $c(\Gp)=1$ and $c(\Gm)=1$, when $\tau = N-1$ and $n=N$. It can also
 happen that $n$ is much larger than $\tau$. Consider, for example, an
 even cycle with edges of alternating sign. In this case $n=N$ --- all
 vertices have edges of both type --- while $\tau=1$ since there is
 only one linearly independent loop.
\end{remark}
   
\begin{example}
  The graph depicted in Figure~\ref{fig:asymptoticGraphs} has a
  flexibility of $\tau = 9 + 1 - 4 -3=3$.  This graph has eight
  vertices which have edges of both signs, and thus eight Gershgorin
  discs that contain the origin in the interior.

In the second graph in figure \ref{fig:motivate} there are three
vertices that have edges of both types, and thus three Gershgorin
discs that contain the origin as an interior point. The flexibility of
this graph, however, is zero so that the number of positive, negative
and zero eigenvalues is fixed and does not vary with the edge weights.

\end{example}

\section{Applications and numerical computations}\label{sec:apps}

\subsection{Random graphs and bifurcations}\label{sec:random}
\begin{figure}[ht]
\begin{centering}
  \includegraphics[width=5in]{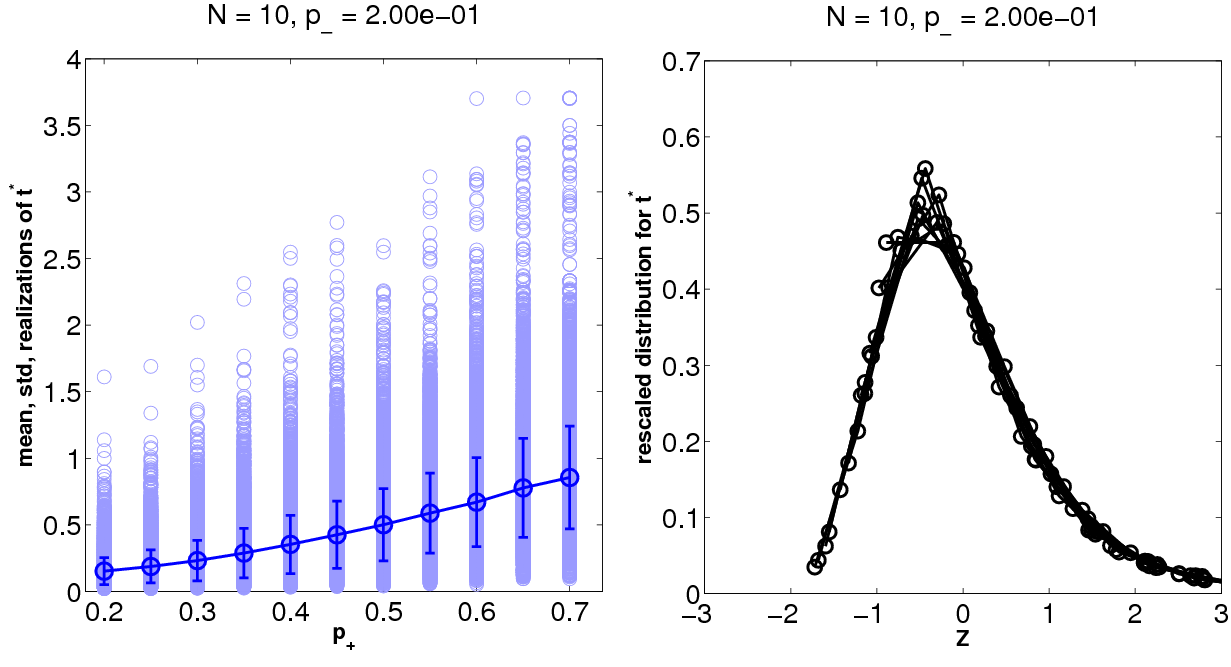}
  \caption{Each of the rows corresponds to a fixed $N$; the top row is
    $N=10$ and the bottom row is $N=50$.  In each row, in the left
    panel, we plot in solid blue the mean and standard deviation for
    ensembles of random matrices where we vary $p_+$, and in light
    blue, we plot the value of $\tstar$ for each individual matrix.
    In the right panel, we have plotted the distributions for all
    values of $p_+$, but rescaled to have mean zero and variance one.
    For the case of $N=50$, we compare these rescaled distributions to
    the standard Gaussian.}
  \label{fig:randomtstar10}
\end{centering}
\end{figure}

\begin{figure}[ht]
\begin{centering}
  \includegraphics[width=5in]{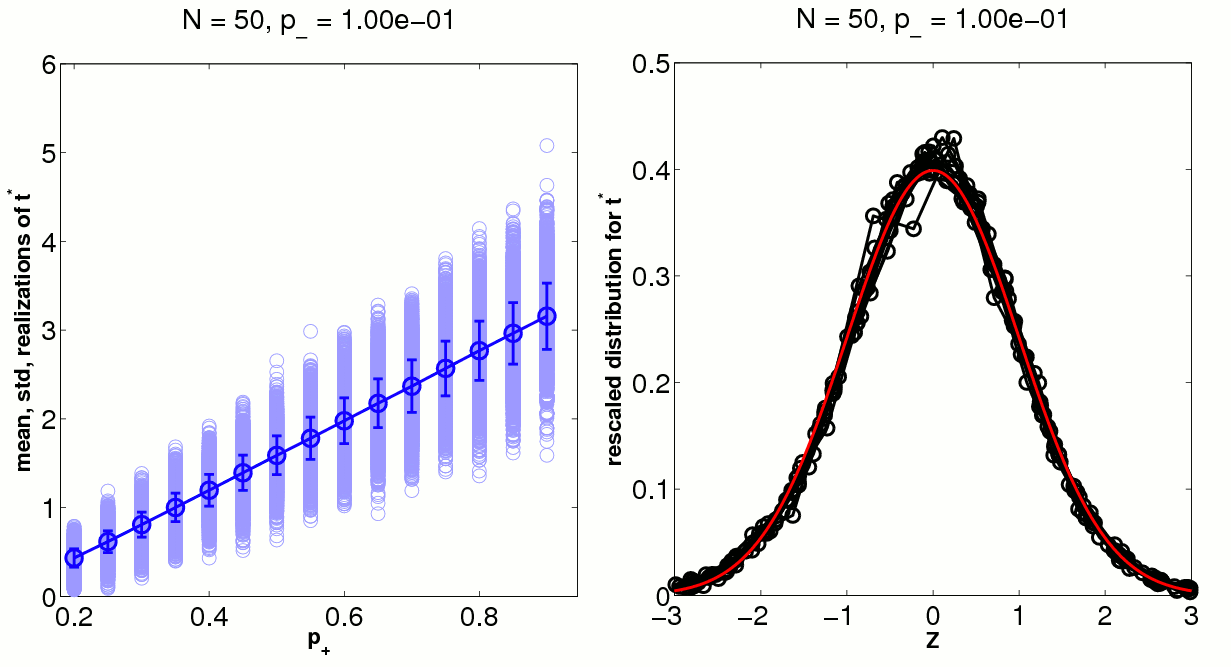}
  \caption{In the left panel, in solid blue we plot the mean and
    standard deviation for an ensemble of random matrices with a given
    $p_+$, $p_- = 0.1$, and $N=50$, versus $p_+$.  In light blue, we
    plot the value of $\tstar$ for each individual matrix.  In the
    right, we have plotted the distributions for all values of $p_+$,
    but rescaled to have mean zero and variance one.  We compare with
    the Gaussian in red.}
  \label{fig:randomtstar50}
\end{centering}
\end{figure}

It was shown in Section~\ref{sec:dct} that for any signed $\Gamma$,
\begin{equation*}
  \tstar := \sup_{t\ge 0} \{t: n_+(\G(t)) = 0\}
\end{equation*}
satisfies $\tstar\in[0,\infty]$, and that $\tstar = 0$ if and only if
$\Gp$ is not connected, and $\tstar = \infty$ if and only if $\Gm =
0$.  Thus we can think of $\tstar$ as a map from the set of finite
graphs to $[0,\infty]$.  For any collection $\mathcal{G}$ of graphs
and a probability measure $\mathbb{P}$ on $\mathcal{G}$, this induces
a random variable $T\colon\mathcal{G}\to[0,\infty]$.

We present some numerically-computed distributions in
Figure~\ref{fig:randomtstar10} and~\ref{fig:randomtstar50} below.  The
random ensemble of matrices is a signed generalization of the
classical Erd\H{o}s--R\'{e}nyi random graph $G(n,p)$ whose specific
definition is given as follows: we fix $N$ and $p_\pm \in[0,1]$.  For
each $1\le i<j\le N$, choose $X_{ij}, Y_{ij}$ independently, with
\begin{equation*}
  \P(X_{ij} = 1) = p_+, \quad \P(X_{ij} = 0) = 1-p_+, \quad \P(Y_{ij} = 1) = p_-, \quad \P(Y_{ij} = 0) = 1-p_-,
\end{equation*}
and we set $\gamma_{ij} = X_{ij} - Y_{ij}$.  From this, we have
\begin{equation*}
  \P(\gamma_{ij} = +1) = p_+(1-p_-),\quad \P(\gamma_{ij} = -1) = p_-(1-p_+),
\end{equation*}
and $\gamma_{ij}= 0$ otherwise, and of course the $\gamma_{ij}$ are
independent as well.  We then set $\gamma_{ij}$ with $i>j$ by
symmetry.  This gives a random distribution on the set of symmetric
graphs with $N$ vertices.  We condition on $\Gp$ being connected and
$\Gm\neq 0$, and then compute the distribution of $\tstar$ over this
ensemble.

We performed a series of numerical experiments on these random
variables, both for $N=10$ (``small matrices'') and $N=50$ (``large
matrices''), the numbers 10 and 50 being chosen arbitrarily.  In each
of these cases, we fixed $p_-$ for all simulations (we chose $p_- =
0.20$ for $N=10$ and $p_- = 0.10$ for $N=50$) and varied $p_+$ over a
range.  For each choice of $p_\pm$, we chose a random ensemble of
$10^4$ matrices using the rules above, and computed $t^*$ for each of
the matrices in the ensemble using a bisection method.  We present the
findings in Figures~\ref{fig:randomtstar10}
and~\ref{fig:randomtstar50}.  We first observe that $t^*$ tends to
increase as $p_+$ increases, which makes sense: adding more positive
edges will make the Laplacian more stable, and higher $p_+$ values
tend to give more positive edges.  What is perhaps surprising is that
the ensemble mean is very close to a linear function of $p_+$ (at
least for a certain range for $N=10$, and for all $p_+$ for $N=50$).
Moreover, we find that there seems to be a universal scaling
distribution for the different values of $p_+$; in each case, what we
do is consider the distribution of $t^*$ for each choice of $p_\pm$,
then normalize this distribution to have mean zero and variance one,
and plot these on top of each other.  For the small ($N=10$) case, the
rescaled distributions overlap well and resemble a lognormal plot; for
the $N=50$ case, the rescaled distributions overlap well and resemble
a normal plot.

To make these observations more precise, in each case we chose $p_+ =
0.45$ and made a QQ-plot of the ensemble distribution versus either a
lognormal (for $N=10$) or a normal (for $N=50$), and these match quite
well.  See Figure~\ref{fig:qq}.

\begin{figure}[ht]
\begin{centering}
  \includegraphics[width=5in]{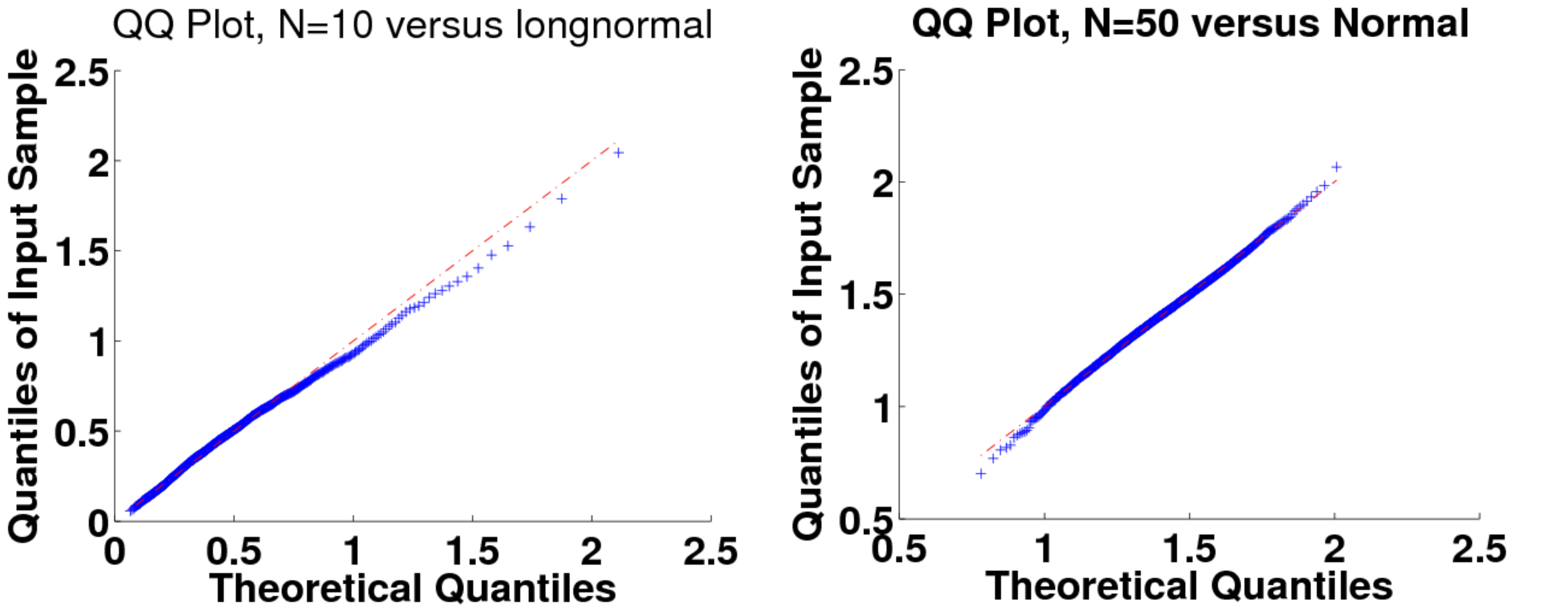}
  \caption{QQ-plots of ensemble distributions.  In the left frame, we
    are plotting a QQ-plot of the ensemble distribution $N=10$, $p_- =
    0.20$, $p_+ = 0.45$ versus a lognormal; in the right frame, we
    plot the ensemble distribution for $N=50, p_- = 0.10, p_+ = 0.45$
    versus a normal.}
  \label{fig:qq}
\end{centering}
\end{figure}

\subsection{Feuds in social networks}\label{sec:social}

We consider two datasets from social networks and compute the
flexibility of the graphs and, in one case, the bifurcations.  Our
computations were facilitated by the {\texttt{matlab\_bgl}}
library~\footnote{\url{http://www.mathworks.com/matlabcentral/fileexchange/10922}}.

The first dataset we consider is from Read~\cite{Read.54}, and
represents sympathetic and antagonistic relationships amongst sixteen
sub-tribes of the Gahuku-Gama people in the highlands of New Guinea:
Gaveve, Kotuni,Gama, Nagamidzhuha, Seu've, Kohika, Notohana, Uheto,
Nagamiza, Masilakidzuha, Asarodzuha, Gahuku, Gehamo, Ove, Ukudzuha and
Alikadzuha.  Warfare between subtribes in this society was very
common. While temporary alliances are common this dataset represents
traditional relationships between subtribes. The positive, or hina,
edges represent subtribes that are traditionally close
politically. Warfare between these subtribes occurs but is limited. It
is understood to be short term condition and is often resolved by
payment of blood money or other concessions. The negative, or rova,
edges represent relations between subtribes that are traditionally
antagonistic. Warfare between these subtribes is much less
constrained.  This has become a somewhat popular data set to analyze:
see the pioneering work of Hage and Harary~\cite{Hage.Harary.book} and
the recent work of Kunegis et. al~\cite{Kunegis2010SAO}.  The graph
has $c(\Gp)=2$ and $c(\Gm)=3$, giving $\tau = 12$. The first component
(A) of $\Gp$ consists of a collection of four tribes (Gaveve,
Kotuni,Gama, Nagamidzhuha; here numbered 1-4) all of whom have
friendly relations. The second component (B) consists of the remaining
twelve tribes (Seu've, Kohika, Notohana, Uheto, Nagamiza,
Masilakidzuha, Asarodzuha, Gahuku, Gehamo, Ove, Ukudzuha and
Alikadzuha; here numbered 5-16), all of whom are connected by at least
one chain of sympathetic relationships.  Relations are more
complicated within component (B) than in component A but there are two
main features to be noted. The first is that removing the
Masilakidzuha (tribe number 10) splits this component into two
subcomponents, the first subcomponent (B1) consisting of the Seu've,
Kohika, Notohana, Uheto, Nagamiza and the second subcomponent (B2)
consisting of the Asarodzuha, Gahuku, Gehamo, Ove, Ukudzuha and
Alikadzuha. There are numerous antagonistic relationships between
tribes in subcomponents B1 and subcomponent B2 but there are no
antagonistic relationships within subcomponents B1 or B2. This
suggests the possibility of a rift developing within component B.

The second dataset is from the Slashdot Zoo.
Slashdot~\footnote{\url{http://www.slashdot.com}} is a user-run
website where links are submitted and voted on by the userbase, and,
furthermore, users can make comments on the links and these comments
are also voted on by the individual users.  A significant amount of
discussion occurs on this website, sometimes friendly and sometimes
fractious, and thus it is not hard to imagine that connections, both
positive and negative, form between users.  Each user is allowed to
tag other users in the database as a ``fan'' or a ``foe'', i.e. if
user $i$ likes the types of comments made by user $j$, user $i$ has
the possibility to tag user $j$ and become a ``fan''.  If this occurs,
then in the Slashdot
Zoo~\footnote{\url{http://konect.uni-koblenz.de/networks/slashdot-zoo},
  dump of userbase, May 2009} database, the edge $(i,j)$ is given
weight $+1$.  Similarly, user $i$ can be a ``foe'' of user $j$ and
this adds a $-1$ on edge $ (i,j)$.  This dataset contained connection
data on 82,144 users, with 549,202 edges in the graph. This network is
not {\em a priori} symmetric, since the fan/foe operations have
directionality, so we imposed symmetry.  Given users $i$ and $j$, if
$i$ and $j$ are both fans of each other, or $i$ is a fan to $j$ and
$j$ is neutral to $i$, then we placed a $+1$ in edge $(i,j)$.
Similarly for foes, if $i$ and $j$ are both foes, or $i$ is a foe to
$j$ and $j$ is neutral to $i$, then we placed a $-1$ in edge $(i,j)$.
In short, we allowed simply extended unidirectional relationships to
be bidirectional as long as the other direction was neutral.  The only
case to think about is what one should do if the two directions are of
opposite sign, i.e. if $i$ was a fan of $j$ and $j$ a foe of $i$; in
this case we decided to assume that the relationship canceled and
placed a 0 on edge $(i,j)$.  As one can imagine, this is relatively
rare, and this happened only 1,949 times in this dataset.

Thus, in both cases, we are considering a network with positive and
negative weights.  Since it has a smaller vertex set, the PNG dataset
was easier to analyze.  We present this data in Figure~\ref{fig:png}.
Since $c(\Gp) = 2$, and $\tau = 12$ the results of Theorem~\ref{thm:main}
imply that we will have  we will have
\begin{equation*}
  \lim_{t\to0^+} n_+(\G(t)) = 1,\quad   \lim_{t\to\infty} n_+(\G(t)) = 13.
\end{equation*}
A symbolic computation using Mathematica gives the crossing polynomial as 
\begin{align*}
  {\mathcal M}(t) &= -45432223 t^{13}+657635624 t^{12}-4187415940 t^{11}+15505043366
  t^{10}-37159886129 t^9\\
  &\quad+60647687776 t^8-68960526571 t^7+54844706645
  t^6-30103762121 t^5+11015925656 t^4\\&\quad-2508107376 t^3+308319872 t^2-14192640
  t.
\end{align*}
The twelve positive roots of the crossing polynomial range from $t\approx.1$ 
to $t\approx 3$. For small positive $t$ we have one positive eigenvalue. 
This, of course, represents the mutual antagonism between the subtribes in component A
and the subtribes in component B. The first bifurcation occurs around $t \approx .1$ when a 
secondary instability develops. Intuitively one expects that this corresponds to 
a split in component $B$, with antagonism between sub-components B1 and B2. 
This is confirmed by the numerics. If we compute the spectrum of the Laplacian 
at the first bifurcation point  $t\approx .988$ there is a single positive 
eigenvalue with eigenvector 
\[
 \vec v = (-.43,-.43,-.44,-.43,.14,.16,.16,.15,.16,.13,.16,.15,.14,.14,.13,.12)^t.
\]
This obviously represents the mutual aggression between components A and B.  
There is a second linearly independent vector in the kernel of the Laplacian 
representing the emerging instability. This eigenvector is given by 
\[
\vec v = (.03, -.002 , .01 , -.003 , .34 , .42 , .38
 , .30 , .22 ,  -.12 , -.27 , -.29 , -.29 , -.24 , -.24 , -.25)^t.
\]
We see that for this mode component $A$ is only weakly involved - component A has only about 3\% of the mass of the 
eigenvector - and the eigenvector clearly describes a rift in component B. The tribes in subcomponent 
B1 (subtribes 5-9: Kohika, Notohana, Seu've, Uheto and Nagamiza) move in one direction and the tribes in subcomponent 
B2 (subtribes 11-16: Asarodzuha, Ove, Gahuku, Gehamo, Ukudzuha and Alikadzuha) move in the opposing
direction. The Masilakidzuha (subtribe 10) have somewhat stronger ties to subcomponent B2 than to subcomponent B1 
(five sympathetic relationships with tribes of B2 vs. two sympathetic relationships with tribes in B1) 
and so move with subcomponent B2. 

The next bifurcation, which occurs around $t\approx .57$, describes a somewhat less obvious conflict. 
The associated null-vector at the bifurcation point is 
\[
 \vec v = (0.44, 0.29, -0.26, -0.37, -0.05, 0.16, 
-0.002, 0.19, -0.23, 0.024, 0.25, -0.04, 
0.32, -0.29, -0.022, -0.40)^t
\]
which describes a conflict with subtribes 1,2,6,8,11 and 13 (Gaveve, Kotuni, Notohana, Uheto, Asarodzuha and Gahuku) forming one 
faction and subtribes 3,4,9,14 and 16 (Gama, Nagamidzhuha, Nagamiza, Gehamo and Alikadzuha) forming another, with much smaller 
involvement of the remaining tribes.
As $t$ increases we see additional eigenvalue crossings, leading to additional conflicts, up to the maximum of thirteen 
unstable modes for $t\approx 3$.  Moreover, if we assume that friend and
foe links are of equal strength, then we have $n_+(\G) = n_+(\G(1)) =
7$. 

\begin{figure}[ht]
\begin{centering}
  \includegraphics[width=5in]{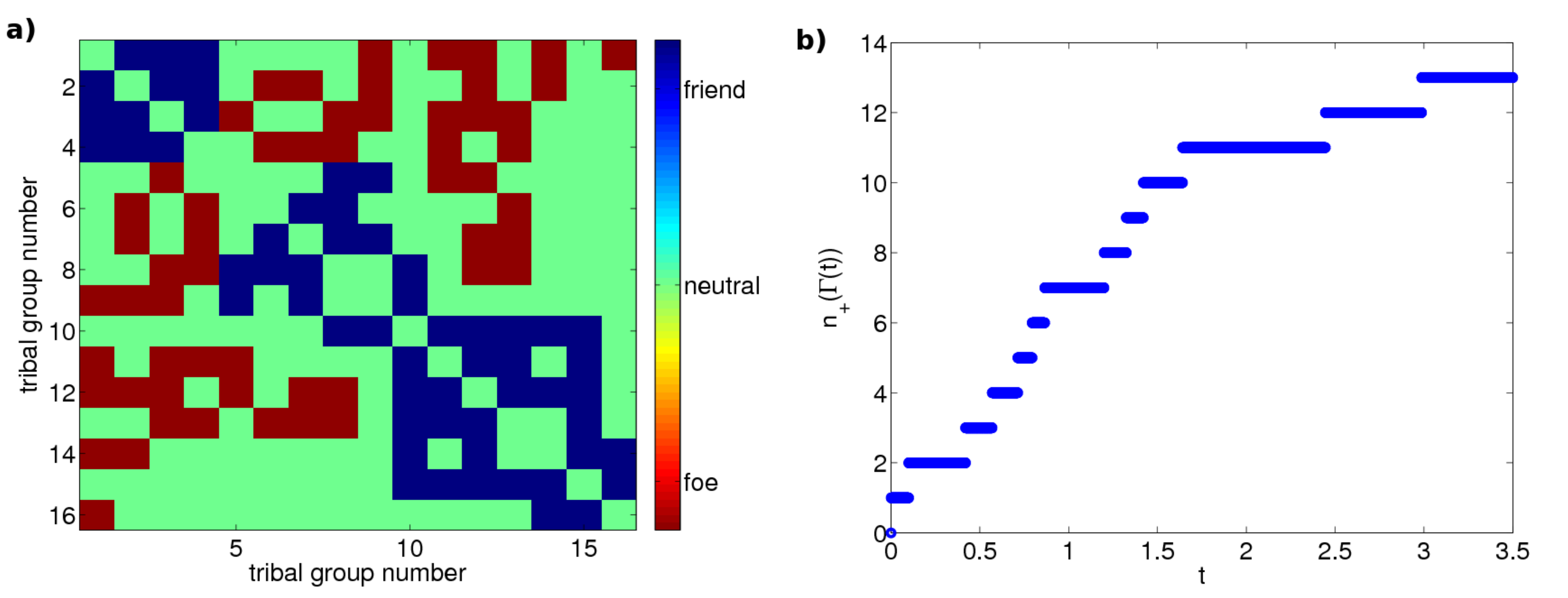}
  \caption{The PNG dataset from~\cite{Read.54}.  In frame (a), we are
    plotting the interaction matrix amongst the 16 tribes.  Blue
    pixels correspond to entries in $\Gp$, red to entries in $\Gm$.
    In frame (b), we plot $n_+(\G(t))$ as a function of $t$; we know
    that this is positive for all $t>0$, since $c(\Gp)>0$, and we see
    that it increases to its maximum value around $t=3.0$. }
  \label{fig:png}
\end{centering}
\end{figure}

For the Slashdot Zoo data set, we concentrated on the component of
the Slashdot database that is friendly to CmdrTaco, the founder of the
site and user number 1.  More specifically, we considered only those
users for which there existed a ``friendly path'' from that user to
CmdrTaco.  This subset was, not surprisingly, the largest connected
component of the full network, and contains 23,514 users.  We then
considered the subgraph of these 23,514 users to itself, and this is
what we define $\G$ to be.  This graph contains 415,118 positive edges
and 117,024 negative edges, and we computed components.  By
definition, $c(\Gp) = 1$, and we compute that $c(\Gm) =$ 10,327,
giving $\G$ a flexibility index of 13,187.  This is, admittedly, a
one-off calculation, but we conjecture that data from social networks
will show this pattern, that even amongst a group of ``friends'', or
common ``fans'' of a particular user, there will be a large number of
instabilities in exactly this manner.



\section*{Acknowledgments}

The authors would like to thank Alejandro Dom\'{\i}nguez-Garc\'{\i}a, 
Susan Tolman and Renato Mirollo for comments and suggestions that improved this work.
JCB was supported in part by NSF grant DMS--1211364 and by a Simons Foundation fellowship.
LD was supported in part by NSF grant CMG--0934491. JCB would also like to thank the 
Mathematics department at MIT and the Applied Mathematics department at Brown for their hospitality
during part of this work. 

\bibliographystyle{plain} \bibliography{Signed,Kuramoto}

\end{document}